\documentclass[12pt]{amsproc}

\usepackage[cp1251]{inputenc}
\RequirePackage{srcltx}
\usepackage{amsmath}
\usepackage{tikz}
\usepackage{verbatim}

\usepackage{amssymb}
\usepackage{amsxtra}

\usepackage{latexsym}
\usepackage{ifthen}
\usepackage{mathrsfs}

\def\N{{{\Bbb N}}}
\def\Z{{{\Bbb Z}}}
\def\T{{{\Bbb T}}}
\def\R{{\Bbb R}}
\def\C{{\Bbb C}}
\def\l{{\lambda }}
\def\a{{\alpha }}
\def\D{{\Delta }}

\def\a{{\alpha}}
\def\b{{\beta}}
\def\d{{\delta}}
\def\e{{\varepsilon}}
\def\s{{\sigma}}
\def\vp{{\varphi}}
\def\t{{\theta }}
\def\g{{\gamma }}

\def\w{{\omega }}

\def\){\right)}
\def\({\left(}

\def\supp{\operatorname{supp}}

\def\sign{\operatorname{sign}}
\def\Re{\operatorname{Re}}

\numberwithin{equation}{section}
\newtheorem{theorem}{Theorem}

\newtheorem{corollary}[theorem]{Corollary}
\newtheorem{lemma}[theorem]{Lemma}

\newtheorem{remark}{Remark}

\newtheorem{property}{Property}

\def\R{\Bbb R}

\def\XXint#1#2#3{{\setbox0=\hbox{$#1{#2#3}{\int}$}
     \vcenter{\hbox{$#2#3$}}\kern-.5\wd0}}

\par

\bigskip
\bigskip

\begin{document}

\title[Properties of moduli of smoothness in $L_p(\R^d)$
]{
Properties of moduli of smoothness in $L_p(\R^d)$}

\author{Yurii
Kolomoitsev$^*$}
\address{
Yu.
Kolomoitsev,
Universit\"at zu L\"ubeck,
Institut f\"ur Mathematik,
Ratzeburger Allee 160,
23562 L\"ubeck, Germany
and Institute of Applied Mathematics and Mechanics of NAS of Ukraine,
General Batyuk Str.~19, 84116 Slov’yans’k, Ukraine
}
\email{kolomoitsev@math.uni-luebeck.de}

\author{Sergey Tikhonov}
\address{S.~Tikhonov, Centre de Recerca Matem\`{a}tica\\
Campus de Bellaterra, Edifici C
08193 Bellaterra (Barcelona), Spain;
ICREA, Pg. Llu\'{i}s Companys 23, 08010 Barcelona, Spain,
 and Universitat Aut\`{o}noma de Barcelona.}
\email{stikhonov@crm.cat}

\date{\today}
\subjclass[2010]{
Primary 41A17, 41A63, 26D10; 26D15 Secondary 41A25, 46E35, 26A33}
\keywords{Moduli of smoothness; $K$-functional; realization of the $K$-functional; Jackson-, Marchaud-, and Ulyanov-inequalities; entire functions of exponential type}

\thanks{$^*$Corresponding author}

\bigskip\bigskip\bigskip\bigskip

\bigskip
\begin{abstract}
 In this paper, we discuss various basic properties of   moduli of smoothness of functions from $L_p(\mathbb{R}^d)$, $0<p\le \infty$.
In particular, complete versions of Jackson-, Marchaud-, and Ulyanov-type inequalities are given for the whole range of $p$. Moreover, equivalences between moduli of smoothness
and the corresponding $K$-functionals and the realization concept are proved.
\end{abstract}

\maketitle

\section{Introduction}


\subsection{Goal of the paper}
The subject of this paper is to collect the main properties of moduli of smoothness in $L_p(\mathbb{R}^d)$, $0<p\le \infty$.
Moduli of smoothness are known to be a very useful concept in many areas of analysis and the PDE's.
Several  basic properties are known for a long time.
For example, a version of the inequality for moduli of smoothness of different orders -- the so-called
Marchaud's inequality -- was obtained already in 1927~\cite{mar}.
On the other hand,  some fundamental characteristics of $\omega_{r}(f,\delta)_{p}$  for the whole range of $p$ and $r$
have remained unknown until now.
In particular, the celebrated Jackson inequality is unknown  for  the case  $0<p<1$ in the multidimensional case for non-periodic functions. Moreover, inequalities between moduli of smoothness in different metrics (the so-called Ulyanov-type inequalities), i.e., $\omega_{r}(f,\delta)_{q}$ vs. $\omega_{k}(f,\delta)_{p}$ for $0<p<q\le \infty$, can be found in the literature only in the weak form~\cite{Cohen, DL, devore, diti}.

In this paper, we not only survey some basic well-known properties  such as  Marchaud's inquality but also  obtain the
complete versions of Jackson  and Ulyanov inequalities.
For this we use needed  Nikolskii--Stechkin and Hardy--Littlewood--Nikolskii type inequalities  for entire functions of exponential type.

 We also derive  equivalences between moduli of smoothness and the following characteristics:
 the corresponding $K$-functionals, the realization concept, and the average moduli of smoothness.
All of them turn out to be  important tools in approximation theory and functional analysis.
Here, again, the results are well known for $1<p<\infty$ and only partially known outside this range.


Throughout the paper,
for $F,G\ge 0$,
 we use the notation
$
\, F \lesssim G
$
 for the
estimate
$\, F \le C\, G,$ where $\, C$ is a positive constant independent of
the essential variables in $\, F$ and $\, G$ (usually, $f$, $\d$, and $\s$). 
 If $\, F \lesssim G \lesssim F$, we write $\, F \asymp G$.

Moreover, for $1\le p\le\infty$,  $p'$ is given by
$
\frac{1}{p}+\frac{1}{p'}=1.
$
For any real number $a$,
$[a]$ is the largest integer not greater than $a$ and $a_+=\max\{a,0\}$.


The rest of the paper is organized as follows. In Subsection~1.2 we give the definitions of the moduli of smoothness of integer and fractional order. In Subsection~1.3 we present the main properties of moduli of smoothness and give some historical comments.
In Section~2 we obtain two important inequalities for entire functions: the Nikol'skii--Stechking-Boas-type inequality and the Hardy-Littlewood-Nikol'skii inequality. Section~3 contains the proofs of the main results of the paper.


\subsection{Definition of moduli of smoothness}

As usual, the $\, r$-th order modulus of smoothness of a function $f\in L_p(\mathbb{R}^d)$ 
is defined  by
\begin{equation}\label{def-mod}
\omega_{r}(f,\delta)_{p}:=\sup_{|h|\le \delta }
\| \Delta_h^r f\|_{p},
\end{equation}
where
$$
\Delta_h f(x)= f(x+h) -f(x),\quad \Delta_h^r =
\Delta_h \Delta_h^{r -1},\quad h\in \R^d,\quad d\ge 1,
$$
and  $|h|:=(\sum_{j=1}^d |h_j|^2)^{1/2}$. Throughout the paper, by  $\Vert f\Vert_p$ we mean the (quasi-)norm $\Vert f\Vert_p=\Vert f\Vert_{L_p(\R^d)}$, $0<p \le \infty$.

We also recall the definition of the modulus of smoothness  $\omega_{\a} (f,\delta)_{p}$  of fractional order $\a >0$: 
\begin{equation}\label{def-mod+}
\omega_{\a} (f,\delta)_{p}
 :=\sup_{|h|\le \delta }
 \left\|
\Delta_{h}^{\a} f \right\|_{p}, 
\end{equation}
\index{\bigskip\textbf{Functionals and functions}!$\omega_{\alpha}(f,\delta)_{p}$, modulus of smoothness of fractional order $\alpha$}\label{MODA}
where
\begin{equation}\label{def-mod++}
\Delta_{h}^{\a} f(x) =\sum\limits_{\nu=0}^\infty(-1)^{\nu}
\binom{\a}{\nu} f\,\big(x+(\a-\nu) h\big),\quad h\in \R^d,
\end{equation}\index{\bigskip\textbf{Operators}!$\Delta_h^\alpha$, fractional difference of order $\alpha>0$}\label{DELTAA}
\!\!\!and
$
\binom{\a}{\nu}=\frac{\a (\a-1)\dots (\a-\nu+1)}{\nu!}$,\quad
$\binom{\a}{0}=$1 (see \cite{But, samko, tabeR}). \index{\bigskip\textbf{Numbers, relations}!$\binom{\a}{\nu}$, binomial coefficients}\label{BINOMA}
It is clear  that for integer $\a$ definition
(\ref{def-mod+}) coincides
 with the classical definition (\ref{def-mod}).
Note that in the case of $0<p<1$, since
\begin{equation}\label{binom}
  \sum_{\nu=0}^\infty \Big|\binom{\a}{\nu}\Big|^p<\infty\quad\text{for}\quad \a\in\N\cup \big((1/p-1),\infty\big),
\end{equation}
it is natural to assume that
$\a>(1/p-1)_+$
while defining the fractional modulus of smoothness in $L_p$.

Some basic properties of moduli of smoothness of integer order can be found in~\cite{BeSh}, \cite{DL}, \cite{timan}, and~\cite{TB}.

\subsection{Main properties of moduli of smoothness in $L_p(\R^d)$}
Below we present Properties 1--17 of (fractional) moduli of smoothness. Historical comments are given after each statement.

\begin{property}\label{pr1}
For  $f,f_1,f_2\in L_p(\R^d)$, $0< p\le \infty$, and $\alpha\in \N\cup ((1/p-1)_+,\infty)$, we have: 
\begin{itemize}
\item[{\rm (a)}]
$ \omega_\a(f,\delta)_p$ is a non-negative non-decreasing function of $\delta$ such that
$\lim\limits_{\delta\to 0+} \omega_\a(f,\delta)_p=0;$ 
\item[{\rm (b)}]
$              \omega_\a(f_1+f_2,\delta)_p\le 2^{(\frac1p-1)_+}
\big(     \omega_\a(f_1,\delta)_p+\omega_\a(f_2,\delta)_p\big);
$
 \item[{\rm (c)}]
$              \omega_{\alpha}(f,\delta)_p\le C(\a,p)\Vert f\Vert_p$;
 \item[{\rm (d)}]
$          \Vert f\Vert_p  \le  \lim_{\d\to \infty}\omega_{\alpha}(f,\delta)_p\le C(\a,p)\Vert f\Vert_p$ if $0<p<\infty$.
\end{itemize}
\end{property}


\begin{property}\label{pr4}
Let $f\in L_p(\R^d)$, $0<p\le \infty$, $\a\in \N\cup ((1/p-1)_+,\infty)$, $\l>0$, and $\d>0$. Then
\begin{equation}\label{Run0}
  \w_\a(f,\l \d)_p\lesssim  (1+\l)^{\a+d(\frac1p-1)_+}\w_\a(f,\d)_p.
\end{equation}
Equivalently, for any $0<h<\d$, one has
\begin{equation}\label{eqMonMod}
 \frac{\w_\a(f,\d)_p}{\d^{\a+d(\frac1p-1)_+}}\lesssim \frac{\w_\a(f,h)_p}{h^{\a+d(\frac1p-1)_+}}.
\end{equation}
\end{property}

If $1\le p\le \infty$ inequality~\eqref{Run0} trivially follows from the equivalence between moduli of smoothness and $K$-functionals (see~\eqref{eq.th6.0d++pr}). If $0<p<1$ only the periodic analogue of~\eqref{Run0} was known.
For $d=1$ this was obtained in \cite{rad} (for integer $\a$) and in~\cite{RS3} (for positive $\a$).
For $d\in \N$ and $\a>0$  see~\cite{KT19m}.

\medskip

In the following three properties, we deal with moduli of smoothness of integer order.

\begin{property}\label{mix}
Let $f\in L_p(\R^d)$, $0<p\le\infty$, and $r\in\N$. Then, for any $\d>0$, we have
\begin{equation}\label{mix1}
\w_r\(f,\d\)_p\asymp \sum_{k_1+\dots+k_d=r}\w_{k_1,\dots,k_d}\(f,\d\)_p,
\end{equation}
where $\w_{k_1,\dots,k_d}\(f,\d\)_p$ is the mixed modulus of smoothness, that is,
$$
\w_{k_1,\dots,k_d}\(f,\d\)_p=\sup_{|h|\le \d} \Vert \Delta_{{\rm e}_1 h_1}^{k_1}\dots \Delta_{{\rm e}_d h_d}^{k_d}f\Vert_p
$$
(here $\{{\rm e}_j\}_{j=1}^d$ is the standard basis in $\R^d$).
Moreover, if $1<p<\infty$, then
\begin{equation}\label{mix2}
\w_r\(f,\d\)_p\asymp \sum_{j=1}^d\w_{r}^{(j)}\(f,\d\)_p,
\end{equation}
where $\w_{r}^{(j)} (f,\d)_p$ is the partial modulus of smoothness, that is,
\begin{equation}\label{moddd}
  \w_{r}^{(j)} (f,\d)_p=\sup_{|h|\le \d} \Vert \Delta^{r}_{{\rm e}_j h}f\Vert_p,\quad j=1,\dots,d.
\end{equation}
\end{property}

For equivalence~\eqref{mix1} in the case $1\le p\le \infty$ see~\cite[p.~338]{BeSh}. For periodic functions equivalence~\eqref{mix2} was given in~\cite{TiMod}. Note also that the mixed moduli of smoothness were studied in, e.g., \cite{DeSh, nikol-book, PST, PST2016}.

\medskip

Recall that the homogeneous Sobolev norm is given by
$\Vert f \Vert_{\dot W_p^{r}}=\sum_{\nu_1+\dots+\nu_d=r}\Vert D^\nu f\Vert_p$,
where as usual $D^\nu f=\frac{\partial^\nu}{\partial^{\nu_1}x_1\dots \partial^{\nu_d} x_d}f$.


\begin{property}\label{sob}
Let $f\in L_p(\R^d)$, $1<p<\infty$, and $r\in \N$. Then
\begin{equation}\label{esob}
  \sup_{h>0}\frac{\w_r(f,h)_p}{h^r}\asymp \Vert f\Vert_{\dot W_p^r}.
\end{equation}
\end{property}

This property can be found in, e.g.,~\cite{Kr07} and~\cite{Mil}, see also~\cite{Cw}.

\begin{property}\label{sum}
Let $0< p,q,s\le \infty$, $1/p+1/q=1/s$, and $r\in \N$. Then for any $f\in L_p(\R^d)$ and $g\in L_q(\R^d)$, we have
\begin{equation}\label{sum1}
  \w_r(fg,\d)_s\le \sum_{k=0}^r \binom{r}{k}\w_k(f,\d)_p \cdot\w_{r-k}(g,\d)_q ,
\end{equation}
where $\w_0(f,\d)_p=\Vert f\Vert_p$ and $\w_{0}(g,\d)_q=\Vert g\Vert_q$.
\end{property}

For the proof of inequality~\eqref{sum1} see~\cite{Jo72}; see also~\cite[4.6.12]{TB} for applications of this inequality.

\begin{property}\label{aver}
Let $f\in L_p(\R^d)$, $0<p\le\infty$, $0<q<\infty$, and $r\in\N$. Then, for any $\d>0$,
\begin{equation}\label{aver1}
  \w_r(f,\d)_p\asymp \(\d^{-d}\int_{|h|\le \d} \Vert \D_h^r f\Vert_p^q dh\)^{1/q}
\end{equation}
and if, additionally, $q\le p$, then
\begin{equation}\label{aver1-}
  \w_r(f,\d)_p\asymp \bigg\Vert \bigg(\d^{-d}\int_{|h|\le \d} |\D_h^r f(\cdot)|^q dh\bigg)^{1/q}\bigg\Vert_p.
\end{equation}

Moreover, if $1<p<\infty$ and $\a>0$, then, for any $\d>0$,
\begin{equation}\label{aver2}
  \w_\a(f,\d)_p\asymp \d^{-d}\int_{|h|\le \d} \Vert \D_h^\a f\Vert_p dh
\end{equation}
and 
\begin{equation}\label{aver2+}
  \w_{2r}(f,\d)_p\asymp \bigg\Vert \d^{-d}\int_{|h|\le \d} \D_h^{2r} f(\cdot -rh)  dh \bigg\Vert_p.
\end{equation}

If $d=1$, then equivalence~\eqref{aver2} holds also for any $0<p\le \infty$ and $\a>(1/p-1)_+$.
\end{property}

In the case $1\le p\le \infty$, the first equivalence~\eqref{aver1} is well known, see, e.g.,~\cite[Ch.~6, \S~5]{DL} or \cite[Appendix A]{KaMilXi}; see also~\cite{DeSh}. For periodic functions $f\in L_p(\T)$, $0<p<1$, and $\a\in\big((1/p-1)_+,\infty\big)$, equivalence~\eqref{aver2} was derived in~\cite{K17}. Equivalence~\eqref{aver2+} for functions on $\T^d$ was obtained in~\cite[8.2.9]{TB}.



\begin{property}\label{pr6} {\sc (Marchaud inequality).}
  Let $f\in L_p(\R^d)$, $0<p\le \infty$, $\a\in \N\cup
    ((1/p-1)_+,\infty)$, and $\g>0 $ be  such that $\a+\g\in \N\cup
    ((1/p-1)_+,\infty)$. Then, for any $\d\in (0,1)$, we have
\begin{equation}\label{eq.lemMarchaudMod-pr}
  \w_\a(f,\d)_p\lesssim \d^\a \(\int_\d^1 \(\frac{\w_{\g+\a}(f,t)_p}{t^{\a}}\)^\theta\frac{dt}{t}+\Vert f\Vert_p^\theta\)^\frac1\theta,
\end{equation}
where
\begin{equation}\label{tetha+}
\theta=\theta(p)=\left\{
       \begin{array}{ll}
         \min(p,2), & \hbox{$p<\infty$;} \\
         1, & \hbox{$p=\infty$.}
       \end{array}
     \right.
\end{equation}
Equivalently,
\begin{equation}\label{eq.lemMarchaudModinf}
  \w_\a(f,\d)_p\lesssim \d^\a \(\int_\d^\infty \(\frac{\w_{\g+\a}(f,t)_p}{t^{\a}}\)^\theta\frac{dt}{t}\)^\frac1\theta.
\end{equation}
\end{property}

The Marchaud inequality for the moduli of smoothness of integer order is the classical result in approximation theory (see, e.g.,~\cite[p. 48]{DL} and~\cite{marchaud}). The case $1<p<\infty$ for fractional moduli was handled   in~\cite[Theorem 2.1]{Treb}.

\begin{property}\label{pr6+} {\sc (Reverse Marchaud inequality).}
  Let $f\in L_p(\R^d)$, $0<p\le \infty$, and $\a,\b\in \N\cup
    ((1/p-1)_+,\infty)$. Then
\begin{equation}\label{propodmod}
               \omega_{\alpha+\b}(f,\delta)_p\lesssim \omega_{\b}(f,\delta)_p.
\end{equation}
Moreover, if $1<p<\infty$ and $\a,\b>0$, then, for any $\d \in (0,1)$, we have
\begin{equation}\label{SJJ}
  \d^\a\bigg(\int_\d^1 \( \frac{\w_{\a+\b}(f,t)_p}{t^{\a}}\)^\tau \frac{dt}{t}   \bigg)^{\frac1\tau}
\lesssim \w_\b\(f,\d\)_p,
\end{equation}
where $\tau=\max(p,2).$
\end{property}

Inequality~\eqref{propodmod} easily follows from~\eqref{binom}. Inequalities of type~\eqref{SJJ} in the one-dimensional periodic case were first obtained by Timan in~\cite{Tim66}. The general case was developed in~\cite{ddt}.


\begin{property}\label{pr7} {\sc (Sharp Ulyanov inequality).}
    Let $f\in L_p(\R^d)$, $0<p<q\le\infty$,  $\a \in \N\cup ((1-1/q)_+,\infty)$, and $\g\ge 0$ be such that $\a+\g \in \N\cup ((1/p-1)_+,\infty)$.
    Then, for any $\d \in (0,1)$, we have
\begin{equation}\label{ulpr}
         \w_\a(f,\d)_q\lesssim \(\int_0^\d   \(
   \frac{\w_{\a+\g}(f,t)_p}{t^\g}\eta\(\frac 1t\)
 \)^{q_1}\frac {dt}{t}\)^\frac 1{q_1}+\d^\a \Vert f\Vert_p,
\end{equation}
where
$$
q_1:=\left\{
      \begin{array}{ll}
        q, & \hbox{$q<\infty;$} \\
        1, & \hbox{$q=\infty$}
      \end{array}
    \right.
$$
and

{\rm 1)} if $0<p\le 1$ and $p<q\le\infty$, then
\begin{equation}\label{ulpr1}
  \eta(t)
:=\left\{
         \begin{array}{ll}
           t^{d(\frac1p-1)}, & \hbox{$\g> d\(1-\frac1q\)_+$}; \\
           t^{d(\frac1p-1)}, & \hbox{$\g=d\(1-\frac1q\)_+\ge 1$, $d\ge 2$, and $\a+\g\in \N$}; \\
           t^{d(\frac1p-1)}\ln^\frac1{q_1} (t+1), & \hbox{$\g=d\(1-\frac1q\)_+\ge 1$, $d\ge 2$,  and $\a+\g\not\in \N$}; \\
           t^{d(\frac1p-1)}\ln^\frac1{q} (t+1), & \hbox{$0<\g=d\(1-\frac1q\)_+=1$ and $d=1$}; \\
           t^{d(\frac1p-1)}\ln^\frac1{q} (t+1), & \hbox{$0<\g=d\(1-\frac1q\)_+<1$}; \\
           t^{d(\frac1p-\frac1q)-\g}, & \hbox{$0< \g<d\(1-\frac1q\)_+$};\\
           t^{d(\frac1p-\frac1q)}, & \hbox{$\g=0$,}
         \end{array}
       \right.
\end{equation}

{\rm 2)} if $1<p< q\le\infty$, then
\begin{equation}\label{ulpr2}
  \eta(t):=
\left\{
         \begin{array}{ll}
           1, & \hbox{$\g\ge d(\frac1p-\frac1q),\quad q<\infty$}; \\
           1, & \hbox{$\g> \frac dp,\quad q=\infty$}; \\
           \ln^\frac1{p'} (t+1), & \hbox{$\g=\frac dp,\quad q=\infty$}; \\
           t^{d(\frac1p-\frac1q)-\g}, & \hbox{$0\le \g<d(\frac1p-\frac1q)$}.\\
         \end{array}
       \right.
\end{equation}
Moreover, the term $\d^\a\Vert f\Vert_p$ in~\eqref{ulpr} can be dropped  if any of the following conditions holds:
\begin{equation*}
      \left\{
      \begin{array}{ll}
        \g=0, & \hbox{$0<p<q\le 1;$} \\
        0\le \g <d(1-\frac1q), & \hbox{$0<p\le 1<q\le \infty;$} \\
        \g=1 , & \hbox{$d=1$, $0<p\le 1<q=\infty;$} \\
        \g =d(1-\frac1q)\ge 1, & \hbox{$0<p\le 1<q\le \infty$, $d\ge 2$, $\a+\g\in \N;$}\\
        0\le \g\le d(\frac1p-\frac1q), & \hbox{$1<p<q<\infty;$} \\
        0\le \g<\frac dp, & \hbox{$1<p<q=\infty.$}
      \end{array}
    \right.
\end{equation*}
\end{property}

In the case $\g=0$, inequality~\eqref{ulpr} is the classical Ulyanov inequality of different metrics~\cite{U} given by
\begin{equation}\label{UUU}
         \w_\a(f,\d)_q\lesssim \(\int_0^\d   \Big(
   \frac{\w_{\a}(f,t)_p}{t^\g}
 \Big)^{q_1}\frac {dt}{t}\)^\frac 1{q_1},\quad \g=d(1/p-1/q),\, 0<p<q\le \infty.
\end{equation}
%
The detailed historical review can be found in~\cite{KT19m}.
Let us only note that the comprehensive study of~\eqref{UUU} was given in~\cite{diti}.
The sharp Ulyanov inequality in the form
$$
         \w_\a(f,\d)_q\lesssim \(\int_0^\d   \Big(
   \frac{\w_{\a+\g}(f,t)_p}{t^\g}
 \Big)^{q}\frac {dt}{t}\)^\frac 1{q},\quad \g=d(1/p-1/q),\, 1<p<q<\infty,
$$
 was first derived in~\cite{ST} and~\cite{Treb} (see \cite{Ti} for the limiting cases $p=1$ and/or $q=\infty$ for functions on $\T$).
The periodic analogue of~\eqref{ulpr} for any $\g>0$ and $0<p<q\le \infty$ has been recently obtained in~\cite{KT19m}. For various moduli of smoothness  sharp Ulyanov inequalities were also established in \cite{oscar, polina, goga}.

Note that in the proof of Property~\ref{pr7} given in Section~3, we obtain slightly stronger inequalities than those given in~\eqref{ulpr} including some important corollaries. 

Now we concern with the Kolyada-type inequality, which is another improvement of the classical Ulyanov inequality~\eqref{UUU} along with the sharp Ulyanov inequality~\eqref{ulpr}.

\begin{property}\label{Koly} {\sc (Kolyada inequality).}
Let $f\in L_p(\R^d)$, $1<p< q<\infty$,
$\theta=d\(1/p-1/q\)$, and $\a\in \N\cup ((1/p-1)_+,\infty)$, $\a>\t$. Then
\begin{equation}\label{eqth3.1Kmod}
    \d^{\a-\theta}\(\int_\d^\infty
\(\frac{\w_\a(f,t)_{q}}{t^{\a-\theta}}\)^p\frac{dt}{t}\)^\frac1p \lesssim\(\int_0^\d \(\frac{\w_\a(f,t)_{p}}{t^\t}\)^q\frac{dt}{t}\)^\frac1q.
\end{equation}
\end{property}

The periodic analogue of~\eqref{eqth3.1Kmod} was obtained by~Kolyada~\cite{kol}. He also  derived such inequalities for analytic Hardy spaces on the disc in the case  $0<p<q<\infty$. For functions on $\R^d$, inequality~\eqref{eqth3.1Kmod} was proved for Lebesgue spaces in~\cite{Treb} (see also~\cite{gol1}) and for Hardy spaces in~\cite{KT19m}.   
%
%
Note that (\ref{eqth3.1Kmod}) is not valid in $L_p(\R^d)$ spaces for $p=1$, $d=1$ but is true for $p=1$, $d\ge 2$.

%

\begin{property}\label{pr8}
     Let $f\in L_p(\R^d)$, $1 \le p \le \infty$ and $r, m \in \mathbb{N}$. Then
 \begin{equation}\label{MarchaudClassic}
        \d^{-m} \omega_{r+m}(f,\d)_p \lesssim \sup_{|\beta| =m} \omega_r(D^\beta f, \d)_p \lesssim \int_0^\d u^{-m} \omega_{r+m}(f,u)_p \frac{du}{u}.
    \end{equation}
Moreover, if $1<p<\infty$, then
    \begin{equation}\label{inequalTrebels1}
        \sup_{|\beta| =m} \omega_r(D^\beta f, \d)_p \lesssim \left(\int_0^\d (u^{-m} \omega_{r+m}(f,u)_p)^{\theta} \frac{du}{u}\right)^{1/\theta},
    \end{equation}
where $\theta=\min\(p,2\)$.
If, in addition, $f$ is such that $\frac{\partial^m f}{\partial x_j^m} \in L_p(\mathbb{R}^d)$ for $j=1, \ldots, d$, then,
    \begin{equation}\label{inequalTrebels2}
          \left(\int_0^\d (u^{-m} \omega_{r+m}(f,u)_p)^\tau \frac{du}{u}\right)^{1/\tau} \lesssim \sup_{j=1, \ldots, d} \omega_r\left(\frac{\partial^m f}{\partial x_j^m}, \d\right)_p,
    \end{equation}
where $\tau=\max\(p,2\)$.
\end{property}

For inequalities~\eqref{MarchaudClassic} see~\cite[p.~342--343]{BeSh}.
Inequality~\eqref{inequalTrebels1} was proved in~\cite[Theorem 2.3]{Treb}.
Inequality~\eqref{inequalTrebels2} has been recently obtained in~\cite[Theorem~12.2]{DoTi}.
See also~\cite{diti07}, \cite{K18}, and~\cite{KL19} for periodic analogues of~\eqref{MarchaudClassic} and~\eqref{inequalTrebels1} in the case $0<p<1$.

\medskip

To formulate the  next properties, we introduce some notations.  By  $\mathcal{B}_{\s,p}=\mathcal{B}_{\s,p}(\R^d)$, $\s>0$, $0<p\le \infty$, we denote the Bernstein space of entire functions of exponential type~$\s$ (e.f.e.t.). That is, $f \in \mathcal{B}_{\s,p}$ if $f\in L_p(\R^d)\cap C_b(\R^d)$ and $\supp \mathcal{F}(f)\subset B_\s=\{x\in \R^d\,:\, |x|<\s\}$. Here and in what follows,
the Fourier transform of $f\in L_1(\R^d)$ is given by
$$
\widehat{f}(\xi)=\mathcal{F}(f)(\xi)=\int_{\R^d} f(x) e^{-i(x,\xi)} dx.
$$
In the case $0<p<1$, we assume that  $f$ belongs to $\mathcal{S}'(\R^d)$, the space of all tempered distributions on $\R^d$.

Let $E_\s(f)_{p}$ be the best approximation of $f\in L_p(\R^d)$ by e.f.e.t. $P\in\mathcal{B}_{\s,p}$, i.e.,
$$
E_\s(f)_{p}=\inf_{P\in \mathcal{B}_{\s,p}}\Vert f-P\Vert_{p}.
$$
We also set $E_0(f)_p=\Vert f\Vert_p$ for $p<\infty$ and $E_0(f)_\infty=\inf_{c\in \C}\Vert f-c\Vert_\infty$.

\begin{property}\label{pr11} {\sc (Jackson inequality).}
Let $f\in L_p(\R^d)$, $0<p\le\infty$, $\s>0$, and $\a\in\N\cup \big((1/p-1)_+,\infty\big)$. Then
\begin{equation}\label{JacksonSO-pr}
E_\s(f)_p\lesssim \w_\a\(f,\frac1\s\)_p.
\end{equation}
Moreover, if $1<p<\infty$ and $\a>0$, then for any $\s\ge 1$ we have
\begin{equation}\label{SharpJack}
\frac1{\s^\a}\bigg(\sum_{k=1}^{[\s]} (k+1)^{\a \tau-1}E_k(f)_p^{\tau}\bigg)^{\frac1\tau}
\lesssim \w_\a\(f,\frac1\s\)_p,
\end{equation}
where $\tau=\max(p,2).$
\end{property}

For Jackson's inequality~\eqref{JacksonSO-pr} in the case $1\le p\le \infty$ see, e.g.,~\cite[p.~279]{timan}. In the case $0<p<1$, $\a\in \N$, and $d=1$, this inequality was derived in~\cite{Tab81} (see also~\cite{BRS09}). Sharp Jackson inequality~\eqref{SharpJack} was obtained in~\cite{ddt}.
It is worth mentioning that this inequality is equivalent to~\eqref{SJJ}, see \cite{ddt}.
Some historical remarks on Jackson's inequality for periodic functions can be found in \cite{ivanov}.

\medskip


\begin{property}\label{pr12} {\sc (Inverse approximation theorem).}
  Let $f\in L_p(\R^d)$, $0<p\le\infty$, $\a \in \N \cup ((1/p-1)_+,\infty)$, and $\s\ge 1$. Then  we have
\begin{equation}\label{eqconverseMod}
\w_\a\(f,\frac1\s\)_p\lesssim \frac1{\s^\a}\bigg(\sum_{k=0}^{[\s]} (k+1)^{\a \theta-1}E_k(f)_p^{\theta}\bigg)^{\frac1\theta},
\end{equation}
where $\theta=\min(p,2)$ if $p<\infty$ and $\theta=1$ if $p=\infty$.
\end{property}

In the case $1\le p\le \infty$, $\a\in\N$, inequality~\eqref{eqconverseMod} is well known (see, e.g.,~\cite[Ch.~7]{DL},  \cite{dai-d}, \cite{timan} and the references therein). In the case $0<p<1$, $d=1$, and $\a\in \N$, inequality~\eqref{eqconverseMod} was obtained in~\cite{Tab81}.
In other cases, the result seems to be new (cf.~\cite{dit-acta}). The proof is based on the corresponding  Bernstein inequality.



\begin{property}\label{Rathor}
Let $f\in L_p(\R^d)$, $0< p\le \infty$, and $\a\in \N\cup ((1/p-1)_+,\infty)$. The following conditions are equivalent:

$(i)$ for some $\b>\a+(1/p-1)_+$ we have 
$$
\w_\a(f,\d)_p\asymp \w_\b(f,\d)_p\quad\text{for all}\quad \d\in (0,1),
$$

$(ii)$ there holds
\begin{equation*}
\w_\a\(f,\frac1\s\)_p\asymp E_\s(f)_p\quad\text{for all}\quad \s\ge 1.
\end{equation*}
\end{property}

In the case $1\le p\le \infty$ and $\a\in \N$, this property was obtained in~\cite{Rathor} (the case $d=1$) and in~\cite{GIT} (the case $d\ge 1$).
In the case $0<p<1$, Property~\ref{Rathor} is known only for functions $f\in L_p(\T)$ and $f\in L_p[-1,1]$, see, e.g.,~\cite{K07} and~\cite{K18}.

\medskip

In what follows, we will use the directional derivative of $f$ of order $\a>0$ along a vector $\zeta\in \R^d$  given by
$$
D_{\zeta}^{\a} f(x)=\mathcal{F}^{-1}\({(i\xi,\zeta)^\a} \widehat{f}(\xi)\)(x).
$$


\begin{property}\label{pr13}
  Let $f\in L_p(\R^d)$, $0< p\le \infty$, and $\a \in \N \cup ((1/p-1)_+,\infty)$. Then
  \begin{equation}\label{eq7R}
    2^{-n\a} \sup_{|\zeta|=1,\, \zeta\in \R^d}\Vert D_\zeta^\a P_{2^n} \Vert_p
\lesssim \omega_\a(f,2^{-n})_{p}
\lesssim \sum_{k=n+1}^\infty 2^{-k\a} \sup_{|\zeta|=1,\, \zeta\in \R^d}\Vert D_\zeta^\a P_{2^k} \Vert_p,
  \end{equation}
  where $P_{2^k}\in \mathcal{B}_{2^k,p}$, $k\in \N$, are such that $\Vert f-P_{2^k}\Vert_p=E_{2^k}(f)_p$.

Moreover, if  $1<p<\infty$, then
  \begin{equation}\label{eq7-}
\begin{split}
    \(\sum_{k=n+1}^\infty 2^{-k\a\tau}\Vert (-\Delta)^{\a/2}P_{2^k}\Vert_{p}^\tau\)^\frac1\tau&\lesssim \omega_\a(f,2^{-n})_{p}\\
&\lesssim  \(\sum_{k=n+1}^\infty 2^{-k\a\theta}\Vert (-\Delta)^{\a/2}P_{2^k}\Vert_{p}^\theta\)^\frac1\theta,
 \end{split}
\end{equation}
where  $\tau=\max(2,p)$ and $\theta=\min(2,p)$.
\end{property}

Concerning the existence of $P_{2^k}$ in Property~\ref{pr13} see, e.g.,~\cite[Theorem~2.6.3]{timan}.
The above inequalities~\eqref{eq7R} and~\eqref{eq7-} were obtained in~\cite{KT19a}.

\begin{remark}
Note that in the case $1<p<\infty$, the best approximants $P_{2^k}$ in inequalities~\eqref{eq7-} can be replaced by other methods of approximation such as the $\ell_q$-Fourier means with $q=1,\infty$, the de la Vall\'ee Poussin means, or the Riesz spherical means. See details in~\cite{KT19a}.
\end{remark}

\medskip

Now we deal with equivalence results for moduli of smoothness in terms of $K$-functionals and their realizations.   Let us recall that for any $f\in L_p(\R^d)$, $1\le p\le \infty$, and $r\in \N$, we have (see~\cite{johnen} and~\cite[Ch.5]{BeSh})
\begin{equation}\label{mk}
\w_r(f,\d)_p\asymp K(f,\d; L_p(\R^d), \dot W_p^r(\R^d)):=\inf_g\{\Vert f-g\Vert_p+\d^r \Vert g\Vert_{\dot W_p^r}\}.
\end{equation}
Moreover, in the case $1<p<\infty$ and $\a>0$ we have (see~\cite{Wil})
\begin{equation}\label{kwil}
  \w_\a(f,\d)_p\asymp K(f,\d; L_p(\R^d), \dot H_p^\a(\R^d)):=\inf_g\{\Vert f-g\Vert_p+\d^\a \Vert (-\Delta)^{\a/2}g\Vert_{p}\}.
\end{equation}
Here the Sobolev and the Riesz potential spaces are given by
$$
\dot W_p^r(\R^d):=\{f\in \dot{\mathcal{S}}'(\R^d):\Vert f \Vert_{\dot W_p^{r}}=\sum_{|\nu|_{\ell_1}=r}\Vert D^\nu f\Vert_p<\infty\}
$$
and
$$
\dot H_p^\a(\R^d):=\{f\in \dot{\mathcal{S}}'(\R^d):\Vert f \Vert_{\dot H_p^{\a}}=\Vert (-\Delta)^{\a/2} f\Vert_p<\infty\}
$$
respectively, where $\dot{\mathcal{S}}'(\R^d)$ is the space of all continuous functionals on $\dot{\mathcal{S}}(\R^n)$ given by
$$
\dot{\mathcal{S}}(\R^n)=\left\{\vp\in{\mathcal{S}}(\R^n)\,:\, (D^\nu
\widehat{\vp})(0)=0\,\,\,\text{for\ all}\,\,\,\nu\in\N^n\cup\{0\}\right\}.
$$

To obtain analogues of the above equivalences for all $0<p\le \infty$ and $\a>0$, we introduce the following $K$-functional with respect to the  directional derivative: 
$$
\mathcal{K}_\a(f,\d)_p=\inf_{g} \{\Vert f-g\Vert_p+\d^\a \sup_{|\zeta|=1,\, \zeta\in \R^d}\Vert D_\zeta^\a g \Vert_p \}.
$$

\begin{property}\label{pr9}
  Let $f\in L_p(\R^d)$, $1\le p\le \infty$, and $\a>0$. Then, for any $\d\in (0,1)$, we have
\begin{equation}\label{eq.th6.0d++pr}
{\omega}_\a(f,\d)_p\asymp \mathcal{K}_{\a}(f,\d)_{p}.
\end{equation}
\end{property}

This equivalence fails for  $0<p<1$ since in this case $\mathcal{K}_\a(f,\d)_{p}\equiv 0$ (see \cite{DHI}). A
suitable substitute for the $K$-functional for $p<1$ is the
realization concept given by
$$
\mathcal{R}_{\a}(f,\d)_{p}=\inf_{P\in\mathcal{B}_{1/\d,p}}\left\{\Vert
f-P\Vert_p+\d^{\a}\sup_{\zeta\in \R^d,\,|\zeta|=1}\Vert D_{\zeta}^{\a} P\Vert_{p}\right\}.
$$

\begin{property}\label{pr10}
Let $f\in L_p(\R^d)$, $0<p\le \infty$, and $\a\in\N\cup \big((1/p-1)_+,\infty\big)$. Then, for any $\d\in (0,1)$, we have
\begin{equation}\label{eq.th6.0d-pr}
{\omega}_\a(f,\d)_p\asymp \mathcal{R}_{\a}(f,\d)_{p}.
\end{equation}
Moreover, if $P\in\mathcal{B}_{1/\d,p}$ is such that
$
\Vert f-P\Vert_p\lesssim E_{1/\d}(f)_p,
$
then
\begin{equation}\label{eqvwithbest-pr}
  {\omega}_\a(f,\d)_p\asymp \Vert
f-P\Vert_p+\d^{\a}\sup_{\zeta\in \R^d,\,|\zeta|=1}\Vert D_{\zeta}^{\a} P\Vert_{p}.
\end{equation}
In particular, we have for $r\in\N$
\begin{equation}\label{RealW}
{\omega}_r(f,\d)_p\asymp \inf_{P\in\mathcal{B}_{1/\d,p}}\left\{\Vert
f-P\Vert_p+\d^{r}
\Vert P\Vert_{\dot W_p^r}
\right\}.
\end{equation}
Moreover, if $P\in\mathcal{B}_{1/\d,p}$ is such that
$
\Vert f-P\Vert_p\lesssim E_{1/\d}(f)_p,
$
then
\begin{equation*}
  {\omega}_r(f,\d)_p\asymp \Vert
f-P\Vert_p+\d^{r}\Vert P\Vert_{\dot W_p^r}.
\end{equation*}
\end{property}

In the case $1<p<\infty$, equivalence~\eqref{eq.th6.0d-pr} is a combination of results from~\cite{HI} and~\eqref{kwil}.
For the case $0<p\le 1$,  $d=1$, and $\a\in\N$ see \cite{DHI}.

\begin{remark}\label{rem1}
  Note that Properties~\ref{pr1}--\ref{pr10} except Property~\ref{pr1}(d) hold for periodic functions from $L_p(\T^d)$. Observe that for $f\in L_p(\T^d)$, $0<p\le \infty$, the term $\Vert f\Vert_p$ in~\eqref{eq.lemMarchaudMod-pr} and~\eqref{ulpr} can be dropped.
\end{remark}

\section{Useful inequalities for entire functions}

\subsection{Polynomial inequalities of Nikol'skii--Stechkin--Boas--types
}\label{sec3}


The results of this subsection are crucial to verify Properties~\ref{pr4}, \ref{pr12},  \ref{pr13},
and~\ref{pr10}.

\begin{lemma}\label{lemNSB}
Let $0<p\le\infty$, $\a,\s>0$, and $\zeta\in \R^d$,
$0<|\zeta|\le 1/{\s}$. Then, for any
$
P_\s\in \mathcal{B}_{\s,p},
$
we have
\begin{equation}\label{ineqNS3}
\Vert D_{\zeta}^{\a} P_\s\Vert_{p}
\asymp\Vert\Delta_\zeta^\a P_\s\Vert_{p},
\end{equation}
where the constants in this equivalence depend only on $p$, $\a$, and $d$.
\end{lemma}

\begin{proof}
Inequalities~\eqref{ineqNS3} are proved in~\cite{Kud14} using certain estimates for maximal functions. Alternatively, \eqref{ineqNS3} can be shown repeating step-by-step the proof of the corresponding result for trigonometric polynomials given in~\cite[Theorem~3.1]{KT19m}.
\end{proof}

For periodic functions  $f\in L_p(\T)$, $1\le p \le\infty$, equivalence~\eqref{ineqNS3} is the known result by Nikol'skii~\cite{nikoL}, Stechkin~\cite{stechkiN}, and Boas~\cite{Boas}. Detailed historical observations can be found in~\cite[Section~3]{KT19m}.



\begin{remark}\label{rem2}
\emph{(i)} Under conditions of Lemma~\ref{lemNSB}, we have for any $P_\s\in \mathcal{B}_{\s,p}$ and $\d\in (0,\s]$ that
\begin{equation}\label{ineqNS3+++}
  \w_\a (P_\s,\d)_{p}\asymp
\sup_{\zeta\in \R^d,\,|\zeta|=1}\big\Vert
\Delta_{\zeta\delta}^\a
P_\s
\big\Vert_{p}.
\end{equation}
Indeed, applying equivalence~\eqref{ineqNS3} twice, we have
\begin{equation*}
  \begin{split}
      \w_\a (P_\s,\d)_{p}&=\sup_{0<h\le \d}\sup_{|\zeta|=1}\big\Vert
\Delta_{\zeta h}^\a P_\s \big\Vert_{p}
\lesssim \sup_{0<h\le \d}\sup_{|\zeta|=1}\big\Vert
D_{\zeta h}^\a P_\s \big\Vert_{p}\\
&= \sup_{0<h\le \d}h^\a\sup_{|\zeta|=1}\big\Vert D_{\zeta}^\a P_\s \big\Vert_{p}
= \sup_{|\zeta|=1}\big\Vert
D_{\zeta \d}^\a P_\s \big\Vert_{p}\lesssim \sup_{|\zeta|=1}\big\Vert
\Delta_{\zeta\delta}^\a P_\s \big\Vert_{p}.
   \end{split}
\end{equation*}
The inverse estimate is clear.

\emph{(ii)} Note that in the  one-dimensional case, the above equivalences~\eqref{ineqNS3} and~\eqref{ineqNS3+++}  have the following form:
$$
\frac{\w_\a (P_\s,\d)_{L_p(\R)}}{\d^\a}\asymp \frac{\Vert \D_h^\a P_\s\Vert_{L_p(\R)}}{h^\a}\asymp \Vert P_\s^{(\a)} \Vert_{L_p(\R)}
$$
for any $\d,h\in (0,1/\s]$.
\end{remark}


Using Lemma~\ref{lemNSB} and Remark~\ref{rem2}, we obtain the following result.

\begin{corollary}\label{corNSB}
  Let $0<p\le\infty$, $\a, \s>0$, and
$0<\d\le 1/ \s$. Then, for any
$P_\s\in\mathcal{B}_{\s,p}$, we have
\begin{equation}\label{ineqNS3cor}
\sup_{\zeta\in \R^d,\,|\zeta|=1}\Vert D_{\zeta}^{\a} P_\s\Vert_{p}
\asymp \d^{-\a}\w_\a (P_\s,\d)_{p}.
\end{equation}
In particular, this implies
$$
\s^\a \w_\a(P_\s,1/\s)_p\asymp {\d^{-\a}}{\w_\a(P_\s,\d)_p}.
$$
\end{corollary}

Further, using \eqref{ineqNS3cor} and~\eqref{binom}, we obtain the following Bernstein type inequality for fractional directional derivatives of entire functions.

\begin{corollary}\label{corNSBBEr}
  Let $0<p\le\infty$, $\a \in \N \cup ((1/p-1)_+,\infty)$, and $\s>0$. Then, for any
$P_\s\in\mathcal{B}_{\s,p}$, we have
\begin{equation}\label{ineqNS3corBEr}
\sup_{\zeta\in \R^d,\,|\zeta|=1}\Vert D_{\zeta}^{\a} P_\s \Vert_{p}
\lesssim \s^\a \Vert P_\s \Vert_{p}.
\end{equation}
\end{corollary}
Inequality~\eqref{ineqNS3corBEr} is the classical Bernstein inequality for $d=1$ (see, e.g., \cite[Ch.~4]{devore}).
For the fractional $\a$ and $d=1$ see~\cite{Liz} and \cite{BL}. For $1\le p\le\infty$, $\a>0$, and $d\in \N$, inequality~\eqref{ineqNS3corBEr} can be obtained from \cite[Theorem 3]{Wil}.



For moduli of smoothness of integer order, we have the following
Nikol'skii--Stechkin--Boas result.

\begin{corollary}\label{eq++}
  Let $0< p\le\infty$, $r\in \N$, $\s>0$, and $0<\d\le 1/ \s$. Then, for any $P_\s\in \mathcal{B}_{\s,p}$, we have
\begin{equation*}
 \Vert P_\s\Vert_{\dot W_p^r}\asymp \d^{-r}\w_r (P_\s,\d)_{p}.
\end{equation*}
\end{corollary}

\begin{proof}
The proof is the same as the proof of Theorem~3.2 in~\cite{KT19m}. It is based on Lemma~\ref{lemNSB}, Corollary~\ref{corNSB}, and
equivalence~\eqref{mix1}.
\end{proof}

We will also need the following equivalence between the fractional Laplacian and the directional derivatives.
\begin{corollary}\label{lemRAZZZ}
  Let $1<p<\infty$ and $\a>0$. Then, for any $P\in \cup_{\s>0}\mathcal{B}_{\s,p}$, we have
\begin{equation*}
  \sup_{\zeta\in \R^d,\,|\zeta|=1}\Vert D_{\zeta}^{\a} P\Vert_{p}\asymp \Vert (-\D)^{\a/2}P\Vert_{p}.
\end{equation*}
Moreover, if $0< p\le \infty$ and $r\in \N$, then
\begin{equation}\label{eqRAZZZ+}
  \sup_{\zeta\in \R^d,\,|\zeta|=1}\Vert D_\zeta^{r} P \Vert_{p}\asymp \Vert P\Vert_{\dot W_p^r}.
\end{equation}
\end{corollary}

\begin{proof}
The first part follows from~\eqref{ineqNS3cor} and the equivalence $\d^{-\a}\w_\a(P,\d)_p\asymp \Vert (-\D)^{\a/2} P\Vert_p$ for $P\in \mathcal{B}_{1/\d,p}$, $1/\d=\s$ (see~\cite{Wil}).
Equivalence~(\ref{eqRAZZZ+}) follows from Corollaries~\ref{corNSB} and~\ref{eq++}.
\end{proof}

Note that for $1\le p\le \infty$ equivalence \eqref{eqRAZZZ+} can be extended to any $f\in W_p^r(\R^d)$,  $r\in \N$. Namely we have 
$$
 \Vert f\Vert_{\dot W_p^r}\asymp \sup_{\zeta\in \R^d,\,|\zeta|=1}\Vert D_{\zeta}^{r} f \Vert_{p}.
$$
%
%

\subsection{Hardy--Littlewood--Nikol'skii inequalities for entire functions}\label{hardy-inequality-section}\label{sec5}

%


For applications, in particular, to obtain the sharp Ulyanov inequality for moduli of smoothness (Property~\ref{pr7}), it is important to derive the  Hardy-Littlewood-Nikol'skii inequalities for directional derivatives. 
As a simple example of such inequalities, we mention the following estimate:
$$
{\sup\limits_{|\zeta|=1,\,\zeta\in\R^d}\Vert D_{\zeta}^{\a} P_\s\Vert_{q}}\lesssim \eta(\s) {\sup\limits_{|\zeta|=1,\,\zeta\in\R^d}\Vert D_{\zeta}^{\a+\g} P_\s\Vert_{p}}, \quad P_\s\in \mathcal{B}_{\s,p}.
$$
For $\eta(\s)=1$, $\g=d(1/p-1/q)$, and $1<p<q<\infty$, this corresponds to the Hardy-Littlewood inequality for fractional integrals. For $\eta(\s)=\s^{d(\frac1p-\frac1q)}$, $\g=0$, and $0<p<q\le \infty$, this coincides with the Nikolskii inequality (see~\cite{diti})
\begin{equation}\label{eqNIKNIK}
  \Vert P_\s \Vert_q\lesssim \s^{d(\frac1p-\frac1q)} \Vert P_\s \Vert_p,\quad P_\s\in \mathcal{B}_{\s,p}.
\end{equation}

In this subsection, we obtain similar inequalities for the previously unknown cases $0<p\le 1$ and $q=\infty$.

Let us recall the definition of the Besov space $B_{p,q}^s(\R^d)$ (see, e.g.,~\cite{TribF}). We  consider the Schwartz function
$\vp\in \mathscr{S}(\R^d)$ such that $\supp\vp\subset
\{\xi\in\R^d\,:\,1/2\le |\xi|\le 2\}$, $\vp(\xi)>0$ for $1/2< |\xi|<2$, $\vp(\xi)=1$ for $5/4< |\xi|<7/4$, and
\begin{equation*}
    \sum_{k=-\infty}^\infty \vp(2^{-k}\xi)=1\quad\text{if}\quad \xi\neq0.
\end{equation*}
We also introduce the functions $\vp_k$ and $\psi$ by means of the relations
$$
\mathcal{F}\vp_k(\xi)=\vp(2^{-k}\xi)
\quad
\text{and}
\quad
\mathcal{F}\psi(\xi)=1-\sum_{k=1}^\infty \vp(2^{-k}\xi).
$$

We  say that $f\in\mathscr{S}'(\R^d)$\label{SS}
belongs to the (non-homogeneous)  Besov space $B_{p,q}^s(\R^d)$, $s\in\R$, $0<p,q\le\infty$, if
$$
\Vert f\Vert_{B_{p,q}^s(\R^d)}=\Vert
\psi*f\Vert_{p}+\bigg(\sum_{k=1}^\infty 2^{sqk} \Vert
\vp_k* f\Vert_{p}^q\bigg)^{1/q}<\infty
$$\index{\bigskip\textbf{Spaces}!$B_{p,q}^s(\R^d)$}\label{BSPQ}
(with the usual modification in the case $q=\infty$).

\begin{lemma}\label{lemma+}
Let $0<p\le 1$, $1<q<\infty$, $\a>0$, $\g=d(1-1/q)$, $\s\ge 1$, and $\a+\g \neq 2k+1$, $k\in \Z_+$. Then, for any $P_\s\in\mathcal{B}_{\s,p}$, we have
\begin{equation*}
 {\sup\limits_{|\zeta|=1,\,\zeta\in\R^d}\Vert D_{\zeta}^{\a} P_\s\Vert_{q}}
\lesssim \s^{d(\frac1p-1)}\ln^\frac1{q} (\s+1) \sup\limits_{|\zeta|=1,\,\zeta\in\R^d}\Vert D_\zeta^{\a+\g} P_\s \Vert_p +\Vert P_\s\Vert_q.
\end{equation*}
\end{lemma}

 \begin{proof}
Let $n\in \N$ be such that $2^{n-1}\le \s<2^{n}$. We have
$$
P_\s=P_\s*\psi+\sum_{k=1}^{n+1} P_\s* \vp_k.
$$
By Corollary~\ref{lemRAZZZ} and the embedding $B_{1,q}^{d(1-1/q)}(\R^d)\subset L_q(\R^d)$  (see, e.g.,~\cite[2.7.1]{TribF}),  we derive
\begin{equation}\label{HLNdir}
  \begin{split}
      \sup\limits_{|\zeta|=1,\,\zeta\in\R^d}\Vert D_{\zeta}^{\a} (P_\s-P_\s* \psi)\Vert_{q}&\lesssim \Vert (-\Delta)^{\a/2} (P_\s-P_\s* \psi)\Vert_q\\
&\lesssim \(\sum_{k=1}^{n+1} 2^{d(q-1)k} \Vert (-\Delta)^{\a/2} P_\s*\vp_k\Vert_1^q \)^\frac1q \\
&\lesssim \(\sum_{k=1}^{n+1} 2^{d(q-1)k-\g qk} \Vert \mathcal{D}^{\a+\g} P_\s\Vert_1^q \)^\frac1q,
  \end{split}
\end{equation}
where
$$
\mathcal{D}^{\a+\g} P_\s (x):=\sum_{j=1}^d \(\frac{\partial}{\partial x_j}\)^{\a+\g} P_\s(x)=\mathcal{F}^{-1}\(w(\xi) \widehat{P_\s}(\xi)\)(x),
$$
and $w(\xi):=(i\xi_1)^{\a+\g}+\dots+(i\xi_d)^{\a+\g}$. In the last inequality in~\eqref{HLNdir}, we applied the fact that the function
$$
g(\xi)=\frac{|\xi|^\a \vp(\xi)}{(i\xi_1)^{\a+\g}+\dots+(i\xi_d)^{\a+\g}}=\frac{|\xi|^\a \vp(\xi)}{w(\xi)}
$$
is the Fourier multiplier in $L_1(\R^d)$, i.e., $\Vert \mathcal{F}^{-1}(g \widehat{f})\Vert_1\lesssim \Vert f\Vert_1$ for any $f\in L_1(\R^d)$.
To verify this, one can use Theorem~6.8 in~\cite{LST} stating that if $g\in L_1(\R^d)$ and $D^\nu g \in L_p(\R^d)$ for some $1<p\le 2$ and all $\nu\in \{0,1\}^d$, $\nu\neq 0$, then $g$ is an $L_1$-multiplier. Here we observe that
\begin{multline*}
    w(\xi)=\cos\frac{({\a+\g})\pi}{2}\(|\xi_1|^{\a+\g}+\dots+|\xi_d|^{\a+\g}\)\\
    +i\(\sin\frac{({\a+\g}) \pi \sign \xi_1}{2}|\xi_1|^{\a+\g}+\dots+\sin\frac{({\a+\g}) \pi \sign \xi_d}{2}|\xi_d|^{\a+\g}\),
\end{multline*}
which implies that $\Re w(\xi)\neq 0$ for $\xi\neq 0$ and ${\a+\g} \neq 2k+1$, $k\in \Z_+$.

Next, using Nikolskii's inequality, we derive from~\eqref{HLNdir} that
\begin{equation}\label{HLNrr}
  \begin{split}
     \sup\limits_{|\zeta|=1,\,\zeta\in\R^d}\Vert D_{\zeta}^{\a} (P_\s-P_\s* \psi)\Vert_{q}
     &\lesssim \(\sum_{k=1}^{n+1} \Vert \mathcal{D}^{\a+\g} P_\s\Vert_1^q \)^\frac1q\\
&\asymp n^\frac1q\Vert \mathcal{D}^{\a+\g} P_\s\Vert_1 \\
&\lesssim  2^{d(\frac1p-1)n} n^\frac1q \Vert \mathcal{D}^{\a+\g} P_\s\Vert_p\\
&\lesssim \s^{d(\frac1p-1)}\ln^\frac1{q} (\s+1) \sup\limits_{|\zeta|=1,\,\zeta\in\R^d}\Vert D_\zeta^{\a+\g} P_\s \Vert_p.
  \end{split}
\end{equation}
Finally, taking into account Bernstein's inequality~\eqref{ineqNS3corBEr} and Young's convolution inequality, we get
$$
\sup\limits_{|\zeta|=1,\,\zeta\in\R^d}\Vert D_{\zeta}^{\a} (P_\s* \psi)\Vert_{q} \lesssim \Vert P_\s *\psi\Vert_q\lesssim \Vert P_\s \Vert_q.
$$
This and~\eqref{HLNrr} imply the required estimate.
\end{proof}



\begin{lemma}\label{lemVspom} Let $P\in \mathcal{B}_{\s,1}$, $\s>0$.

{\textnormal {(i)}}  For $1/{q^*}=(d-1)/d$, we have
\begin{equation}\label{eqStein}
  \Vert P\Vert_{q^*}\le \frac1d
  \sum_{j=1}^d \left\Vert \frac{\partial P}{\partial x_j}\right\Vert_1.
\end{equation}

{\textnormal {(ii)}}  We have
\begin{equation}\label{eqVspom}
\Vert P\Vert_\infty \le \Vert P\Vert_{\dot W_1^d}.
\end{equation}
\end{lemma}

\begin{proof}
Inequality~\eqref{eqStein} is given in~\cite[pp.~129--130]{Stein}. To prove~\eqref{eqVspom}, we use the equality
$$
P(x)=\int_{-\infty}^{x_1}\dots\int_{-\infty}^{x_d}\frac{\partial^d}{\partial x_1\dots \partial x_d} P(t_1,\dots,t_d) dt_1\dots dt_d
$$
noting that by Theorem~3.2.5 in~\cite{nikol-book} we have $\lim_{|x|\to \infty} P(x)=0$.
\end{proof}

Now, we obtain the sharp version of the Hardy-Littlewood-Nikolskii inequality in the case $\a+\g\in \N$ and $\g\ge 1$, cf. Lemma~\ref{lemma+}.

\begin{lemma}\label{lemPolSob}
Let $0<p\le 1$, $1<q\le\infty$, $d\ge 2$, $\a>0$, $\g=d(1-1/q)\ge 1$, and
$\a+\g\in \N$. 
Then, for any $P_\s\in\mathcal{B}_{\s,p}$, we have
\begin{equation}\label{eqlemMM2Sob}
  {\sup\limits_{|\zeta|=1,\,\zeta\in\R^d}\Vert D_{\zeta}^{\a} P_\s\Vert_{q}}\lesssim \s^{d(\frac1p-1)} {\sup\limits_{|\zeta|=1,\,\zeta\in\R^d}\Vert D_\zeta^{\a+\g} P_\s\Vert_p}.
\end{equation}
\end{lemma}

\begin{proof}
By Nikol'skii's inequality, we have
\begin{equation*}
  \Vert P_\s\Vert_{\dot W_1^{\a+\g}}\lesssim \sup_{|\zeta|=1,\,\zeta\in\R^d}\Vert D_\zeta^{\a+\g} P_\s\Vert_1
\lesssim \s^{d(\frac1p-1)} \sup_{|\zeta|=1,\,\zeta\in\R^d}\Vert D_\zeta^{\a+\g} P_\s\Vert_p
\end{equation*}
and, therefore, by Corollary~\ref{lemRAZZZ},
\begin{equation}\label{eqlemMM2Sob+}
  \frac{\sup\limits_{|\zeta|=1,\,\zeta\in\R^d}\Vert D_{\zeta}^{\a} P_\s\Vert_{q}}{\sup\limits_{|\zeta|=1,\,\zeta\in\R^d}\Vert D_\zeta^{\a+\g} P_\s\Vert_p}
  \lesssim
  {\s^{d(\frac1p-1)}}
  \left\{
  \begin{array}{ll}
    \displaystyle\frac{\Vert (-\D)^{\a/2}P_\s \Vert_{q}}{\Vert P_\s \Vert_{\dot W_1^{\a+\g}}}, & \hbox{$1<q<\infty$;} \\
    \displaystyle\frac{\Vert P_\s \Vert_{\dot W_q^\a}}{\Vert P_\s \Vert_{\dot W_1^{\a+\g}}}, & \hbox{$q=\infty$.}
  \end{array}
\right.
\end{equation}

First, we consider the case $1<q<\infty$.  
%
%
Choose $q^*\in (1,q]$ such that
$$
\g-1=d\(\frac1{q^*}-\frac1q\).
$$
Then,  using the Hardy-Littlewood inequality for fractional integrals and  Lemma~\ref{lemVspom}~(i),
we have
\begin{equation*}
  \begin{split}
\Vert (-\D)^{\a/2}P_\s \Vert_{q} &\lesssim \Vert (-\D)^{(\a+\g-1)/2}P_\s \Vert_{q^*}\lesssim \Vert P_\s\Vert_{\dot W_{q^*}^{\a+\g-1}}\\
&\lesssim \Vert P_\s\Vert_{\dot W_{1}^{\a+\g}},
   \end{split}
\end{equation*}
which together with~\eqref{eqlemMM2Sob+} implies~\eqref{eqlemMM2Sob}.


To obtain~\eqref{eqlemMM2Sob} in the case $q=\infty$, it is enough to note that
by~\eqref{eqVspom}, we have
$$
\Vert P_\s\Vert_{\dot W_\infty^\a}\lesssim \Vert P_\s\Vert_{\dot
W_1^{\a+d}},
$$
which completes the proof.


\end{proof}

%
%

\begin{lemma}\label{lemma+inf}
Let $1<p<\infty$, $\a>0$, $\g=d/p$, and $\s\ge 1$. Then, for any $P_\s\in\mathcal{B}_{\s,p}$, we have
\begin{equation}\label{will+inf}
 {\sup\limits_{|\zeta|=1,\,\zeta\in\R^d}\Vert D_{\zeta}^{\a} P_\s \Vert_{\infty}}
\lesssim \ln^\frac 1{p'} (\s+1) \sup\limits_{|\zeta|=1,\,\zeta\in\R^d}\Vert D_\zeta^{\a+\g} P_\s\Vert_p +\Vert P_\s\Vert_p.
\end{equation}
\end{lemma}

\begin{proof}
The proof is based on the Br\'ezis–Wainger-type inequality (see~\cite{BW}), which states  that for every $f\in H_q^l(\R^d)$, $1\le q\le \infty$, $1<p<\infty$, and $l>d/q$, one has
\begin{equation}\label{BGW}
  \Vert f\Vert_\infty\lesssim \(1+\ln^\frac{1}{p'}(1+\Vert f\Vert_{H_q^l})\)
\end{equation}
provided that $\Vert f\Vert_{H_p^k}\le 1$ with $k=d/p$. Here, $H_p^s(\R^d)$ is the fractional Sobolev space, i.e., $\Vert f\Vert_{H_p^s}=\Vert (I-\D)^{s/2} f\Vert_p<\infty.$

We use inequality~\eqref{BGW} with $f=\frac{P_\s}{\Vert P_\s\Vert_{H_p^\g}}$, $\g=k=d/p$, $q=p$, and $l=d/p+1$. Then the Bernstein inequality implies that
\begin{equation*}
  \begin{split}
      \Vert P_\s\Vert_\infty\lesssim \(1+\ln^\frac{1}{p'}\(1+\frac{\Vert P_\s\Vert_{H_p^{\g+1}}}{\Vert P_\s\Vert_{H_p^\g}}\)\)\Vert P_\s\Vert_{H_p^\g}\lesssim
\ln^\frac{1}{p'}\(1+\s\)\Vert P_\s\Vert_{H_p^\g}.
  \end{split}
\end{equation*}
Next, taking into account that $\Vert P_\s\Vert_{H_p^\g}\asymp \Vert (-\Delta)^{\g/2}P_\s \Vert_{p}+\Vert P_\s \Vert_{p}$, Corollary~\ref{lemRAZZZ}, and the Bernstein inequality,  we derive
\begin{equation}\label{BGW2}
  \begin{split}
      {\sup\limits_{|\zeta|=1,\,\zeta\in\R^d}\Vert D_{\zeta}^{\a} P_\s\Vert_{\infty}}
&\lesssim
\ln^\frac{1}{p'}\(1+\s\)\(\Vert (-\Delta)^{(\g+\a)/2}P_\s \Vert_{p}+\Vert (-\Delta)^{\a/2} P_\s \Vert_{p}\)\\
&\lesssim
\ln^\frac{1}{p'}\(1+\s\)\(\Vert (-\Delta)^{(\g+\a)/2}P_\s \Vert_{p}+\s^\a\Vert P_\s \Vert_{p}\).
  \end{split}
\end{equation}
Finally, using the substitution $P_\s(\cdot)\to P_{\s}(\e \cdot)$ with $\e=\s^{p\a+\d}$, $\d>0$, by homogeneity we see that inequality~\eqref{BGW2} can be written in the following form:
\begin{equation*}\label{BGW3}
  \begin{split}
{\sup\limits_{|\zeta|=1,\,\zeta\in\R^d}\Vert D_{\zeta}^{\a} P_\s\Vert_{\infty}} \lesssim
\ln^\frac{1}{p'}\(1+\s\)\Vert (-\Delta)^{(\g+\a)/2}P_\s \Vert_{p}+\Vert P_\s \Vert_{p},
  \end{split}
\end{equation*}
which by Corollary~\ref{lemRAZZZ} implies~\eqref{will+inf}.
\end{proof}

\section{Proofs of Properties~\ref{pr1}--\ref{pr10}}

Before presenting the proofs of Properties~\ref{pr1}--\ref{pr10}, we discuss their relationship.
First, we would like to mention that in the proofs of some results (namely, Properties~\ref{pr4}, \ref{aver}, \ref{pr7}, and~\ref{pr9}) we use the equivalence of moduli of smoothness and realizations of $K$-functionals given in Property~\ref{pr10}.
Second, to prove Properties~\ref{pr6}, \ref{pr12},  \ref{pr13}, and~\ref{pr10},  the Jackson inequality given in Property~\ref{pr11} is applied.
At the same time, the proof of Property~\ref{pr11} is independent on the mentioned above Properties~\ref{pr4}, \ref{aver}, \ref{pr6}, \ref{pr7}, \ref{pr12}, \ref{pr13}, \ref{pr9}, and~\ref{pr10}.
 The proof of Property~\ref{pr10} is based on the Stechkin-Nikolskii inequality given in Corollary~\ref{corNSB}, and Properties~\ref{pr1} and~\ref{pr11}.


\subsection*{Proof of Property~\ref{pr1}}

The proof of statements (a)
(b), and (c) is standard.
For details see also~\cite{But, samko,  tabeR}. The proof of (d) repeats the proof of (3.6) in~\cite{Mil}.

\subsection*{Proof of Property~\ref{pr4}}
To prove~\eqref{Run0}, we follow the ideas from~\cite[p. 194]{run} and~\cite[Theorem~4.3]{KT19m}.
We will consider only the case $0<p<1$. 
The case $1\le p\le \infty$ follows from
(\ref{eq.th6.0d++pr}). 

Denote $\d=(\d_1,\dots,\d_d)\in \R^d$, $\d_j>0$, $j=1,\dots,d$, and
$$
K_\d(x)=\widehat{\vp_\d}(x),
$$
where $\vp_\d(t)=v\((\d,t)\)$, $v\in C^\infty(\R)$, $v(s)=1$ for $|s|\le 1/2$ and $v(s)=0$ for $|s|>1$.

We can assume that $\l>1$.
Fix $h>0$ and suppose that $|\d|\le h$ and $P\in \mathcal{B}_{1/h,p}$.
Using Property~\ref{pr1} and Lemma~\ref{lemNSB}, we obtain
\begin{equation}\label{Run1}
\begin{split}
    \Vert \D_{\l \d}^\a f\Vert_p^p&\lesssim \Vert f-K_{\l \d}*P\Vert_p^p+\Vert \D_{\l \d}^\a (K_{\l \d}*P)\Vert_p^p\\
    &\lesssim \Vert f-P\Vert_p^p+\Vert P-K_{\l \d}*P\Vert_p^p+\Vert D_{\l\d}^\a (K_{\l \d}*P)\Vert_p^p\\
    &= \Vert f-P\Vert_p^p+\Vert P-K_{\l \d}*P\Vert_p^p+\l^{\a p}\Vert  K_{\l \d}*D_{\delta}^{\a} P \Vert_p^p.
\end{split}
\end{equation}

Now, we will show that
\begin{equation}\label{Run2}
\Vert P-K_{\l \d}*P\Vert_p \lesssim \l^{\a+d(\frac1p-1)} \Vert  D_{\delta}^{\a} P \Vert_p
\end{equation}
and
\begin{equation}\label{Run3}
\Vert  K_{\l \d}*D_{\delta}^{\a} P \Vert_p \lesssim \l^{d(\frac1p-1)} \Vert  D_{\delta}^{\a} P \Vert_p.
\end{equation}
It is easy to see that \eqref{Run2} is equivalent to the following inequality
\begin{equation*}
\Vert A_{\l \d}*P\Vert_p \lesssim \l^{d(\frac1p-1)} \Vert P \Vert_p\quad\text{for all}\quad P\in \mathcal{B}_{1/h, p},
\end{equation*}
where
$$
A_\d(x)=\widehat{\psi_\d}(x)
$$
and
$$
\psi_\d(t)=\frac{(1-\vp_\d(t))
}
{(it,\d)^\a}
v\(2^{-1}{\sqrt{{
(t_1\delta_1)^2+\cdots+(t_d\delta_d)^2}}}
\).
$$
Using the properties of the convolution algebras in $L_p(\R^d)$, $0<p<1$, see, e.g.,~\cite[p.~220]{SS}, and denoting $\textbf{1}=(1,\dots,1)$,  we have
\begin{equation*}
  \begin{split}
      \Vert A_{\l \d}*P\Vert_p&\lesssim \(\prod_{j=1}^d \d_j\)^{1-\frac1p}\Vert A_{\l \d}\Vert_p \Vert P\Vert_p
=\Vert \widehat{\psi_{\l \textbf{1}}}\Vert_{p} \Vert P\Vert_p\\
&=\l^{d(\frac1p-1)}\Vert \widehat{\psi_{\textbf{1}}}\Vert_{p} \Vert P\Vert_p\lesssim \l^{d(\frac1p-1)}\Vert P\Vert_p,
  \end{split}
\end{equation*}
that is, \eqref{Run2} is verified. Inequality~\eqref{Run3} can be obtained similarly.

Combining (\ref{Run1})--(\ref{Run3}), we obtain
\begin{equation*}
  \begin{split}
    \Vert \D_{\l \d}^\a f\Vert_p&\lesssim \Vert f-P\Vert_p+\l^{\a+d(\frac1p-1)} \Vert  D_\delta^\a P\Vert_p\\
    &\lesssim \l^{\a+d(\frac1p-1)} \(\Vert f-P\Vert_p+ |\d|^\a \sup_{|\zeta|=1,\, \zeta\in \R^d}\Vert D_\zeta^\a P \Vert_p\).
    \end{split}
\end{equation*}
Finally,
taking infimum over all $P\in \mathcal{B}_{1/h,p}$ and using~\eqref{eq.th6.0d-pr}, we get
 $$
\Vert \D_{\l \d}^\a f\Vert_p\lesssim \l^{\a+d(\frac1p-1)} \mathcal{R}_\a(f,h)_p \lesssim \l^{\a+d(\frac1p-1)} \w_\a(f,h)_p,
  $$
which yields (\ref{Run0}).
\hfill$\square$

Property 2 immediately implies the following result (see also~\cite{rad} and~\cite{tik2}).


\begin{remark}
{\it Let $f\in L_p(\R^d)$,  $0<p\le \infty$, and $\a\in \N\cup ((1/p-1)_+,\infty)$. Then there exists a
function $\varphi$ such that $\varphi(\delta)\to 0$ as $\delta\to 0$, $\varphi(\delta)$ is nondecreasing, ${\varphi(\delta)}\delta^{-\a-d(1/p-1)_+}$ is nonincreasing, and}
\begin{equation*}
 \varphi(\d) \asymp \omega_\a(f,\d)_p, \quad \d>0.
\end{equation*}
\end{remark}

\begin{proof}
 Let us define $ \varphi(t):=t^\beta
\inf_{0<\xi\le t }\{\xi^{-\beta} \omega_\a(f,\xi)_p \}$  with $\b={\a+d(1/p-1)_+}$. It is not difficult to see that $\vp$ satisfies the required properties.
Clearly, $\varphi(\delta) \le \omega_\a(f,\delta
)_p$. At the same time by  (\ref{eqMonMod}) we have
$$
\omega_\a(f,t )_p \lesssim t^\beta
\inf\limits_{0<\xi\le t }\left\{\xi^{-\beta} \omega_\a(f,\xi)_p
\right\}= \varphi(t).
$$
\end{proof}



\subsection*{Proof of Property~\ref{mix}}
In the case $1\le p\le \infty$, the proof of~\eqref{mix1} is given in~\cite[Lemma~4.11, p.~338]{BeSh} and is based on the representation of the total difference of a function via the sum of the mixed differences (Kemperman's formula).
Following the same proof,    we arrive at equivalence~\eqref{mix1} in the case $0<p<1$ as well.

The proof of~\eqref{mix2} can be obtained repeating the proof of Theorem~8 in~\cite{TiMod}.\hfill$\square$

\subsection*{Proof of Property~\ref{sob}} The proof of~\eqref{esob} follows from the equivalence of the $K$-functional corresponding to the couple $(L_p, \dot W_p^r)$ and the modulus $\w_r(f,\d)_p$ and the closedness of the unit ball of $\dot W_p^r(\R^d)$ in $L_p(\R^d)$, see~\cite{Kr07} and~\cite{Cw}.

\subsection*{Proof of Property~\ref{sum}}
The proof of~\eqref{sum1} follows by applying the H\"older inequality and the following equality, which can be easily verified by induction:
$$
\Delta^r_h (f g)=\sum_{k=0}^r \binom{r}{k} \Delta_h^k f(x) \Delta_h^{r-k}g(x+kh).
$$

\subsection*{Proof of Property~\ref{aver}}
In the case $\a\in \N$, the equivalence in~\eqref{aver1} can be proved
by standard way using the equality
\begin{equation}\label{ave*}
  \D_h^r f(x)=\sum_{k=1}^r (-1)^k \binom{r}{k} [\D_{kt}^r f(x+kh)-\D_{h+kt}^r f(x)],
\end{equation}
which holds for any $t,h\in \R^d$, see, e.g.~\cite[Ch.~6, \S~5]{DL}.

Consider the equivalence in~\eqref{aver1-}. We give the detailed proof of the estimate
\begin{equation}\label{ave1}
  \w_r(f,\d)_p\lesssim \bigg\Vert \bigg(\d^{-d}\int_{|t|\le \d} |\D_t^r f(\cdot)|^q dt \bigg)^{1/q}\bigg\Vert_p.
\end{equation}
The  inverse estimate is clear by Minkowski's inequality under the condition $q\le p$.

Integrating~\eqref{ave*}, we have for any positive $q$ and $\d\ge |h|$ that
\begin{equation}\label{ave2}
  |\D_h^r f(x)|^q\lesssim \sum_{k=1}^r \d^{-d}\int_{|t|\le \d} \(|\D_{kt}^r f(x+kh)|^q +|\D_{h+kt}^r f(x)|^q\) dt.
\end{equation}
Therefore,
\begin{equation}\label{ave3}
  \begin{split}
     \Vert \D_h^r f\Vert_p&\lesssim \sum_{k=1}^r \bigg\Vert \bigg(\d^{-d}\int_{|t|\le \d} |\D_{kt}^r f(x+kh)|^q dt \bigg)^{1/q} \bigg\Vert_p\\
            &\qquad\qquad\qquad\qquad+\sum_{k=1}^r\bigg\Vert \bigg(\d^{-d}\int_{|t|\le \d}|\D_{h+kt}^r f(x)|^q dt \bigg)^{1/q} \bigg\Vert_p\\
            &\lesssim \bigg\Vert \bigg(\d^{-d}\int_{|t|\le (r+1)\d}|\D_{t}^r f(x)|^q dt \bigg)^{1/q} \bigg\Vert_p.
  \end{split}
\end{equation}
Using Property~\ref{pr4}, it is not difficult to see that~\eqref{ave3} implies~\eqref{ave1}.

Now, we prove~\eqref{aver2} for fractional moduli of smoothness. It is clear that it suffices to prove the estimate from above.  Let $P\in \mathcal{B}_{\l_\a/\d, p}$, where $\l_\a\in (0,1)$ will be chosen later. By~\eqref{kwil}, we have
\begin{equation}\label{aver3}
  \w_\a(f,\d)_p\lesssim \Vert f-P\Vert_p+\d^\a \Vert (-\D)^{\a/2} P\Vert_p.
\end{equation}
Suppose that $\Vert f-P\Vert_p\lesssim E_{\l_\a/\d}(f)_p$. Using Jackson's inequality~\eqref{JacksonSO-pr}, Property~\ref{pr4}, and~\eqref{aver1}, we have for an integer $r>\a$ that
\begin{equation}\label{aver4}
\begin{split}
     \Vert f-P\Vert_p&\lesssim \w_r(f,\l_\a^{-1}\d)_p \lesssim \w_r(f,\d)_p \\
     &\lesssim \d^{-d}\int_{|h|\le \d} \Vert \D_h^r f\Vert_p dh \lesssim \d^{-d}\int_{|h|\le \d} \Vert \D_h^\a f\Vert_p dh.
\end{split}
\end{equation}

Let us show that
\begin{equation}\label{aver5}
 \d^\a \Vert (-\D)^{\a/2} P\Vert_p \lesssim \bigg\Vert \d^{-d}\int_{|h|\le \d} \D_h^\a P(\cdot-\a h) dh \bigg\Vert_p.
\end{equation}
We will follow an idea from~\cite[8.2.9]{TB}, see also~\cite{K12}. First, we note that~\eqref{aver5} can be rewritten in the following equivalent form
\begin{equation*}
 \Vert \Lambda_{\d,\a}(P)\Vert_p\lesssim \Vert P\Vert_p\quad\text{for any}\quad P\in \mathcal{B}_{\l_\a/\d,p},
\end{equation*}
where
$$
\Lambda_{\d,\a}(P)(x)=\mathcal{F}^{-1}\(g(\d \xi)\widehat{P}(\xi)\)(x)
$$
and
$$
g(x)=\frac{|x|^\a v(\l_\a^{-1}x)}{\int_{|h|\le 1} (1-e^{i (h,x)})^\a dh},
$$
$v\in C^{\infty}(\R^d)$, $v(x)=1$ for $|x|\le 1$ and $v(x)=0$ for $|x|\ge 2$.
Here and throughout to define the fractional exponents, we use the principal branch of the logarithm.
Thus, to prove~\eqref{aver5} it suffices to show that the function $g$ is a Fourier multiplier in $L_p(\R^d)$.
After some simple calculations, we derive that
\begin{equation}\label{aver7}
\begin{split}
    &\int_{|h|\le 1} (1-e^{i (h,x)})^\a dh=c_d\int_0^1 r^{d-1}\int_{-1}^1 (1-e^{iru|x|})^\a (1-u^2)^{\frac{d-3}{2}}dudr\\
    &=c_d2^{\a+1}\int_0^1 r^{d-1}\int_{0}^1 \cos\frac{\a (ru|x|-\pi)}{2} \(\sin \frac{ru|x|}{2}\)^\a (1-u^2)^{\frac{d-3}{2}}dudr,
\end{split}
\end{equation}
where $c_d=\frac{2\pi^{\frac{d-1}{2}}}{\Gamma(\frac{d-1}{2})}$. Then, we can set $g(0)=\lim_{|x|\to 0} \frac{|x|^\a v(\l_\a^{-1}x)}{\int_{|h|\le 1} (1-e^{i (h,x)})^\a dh}$.
It is also clear that we can always find positive $\l_\a$ such that the integral in~\eqref{aver7} is strictly positive or negative for $0<|x|\le 2\l_\a$.  Therefore, $g$ is a smooth function on $\R^d\setminus \{0\}$ with a compact support, which implies that $g$ is a Fourier multiplier in $L_p(\R^d)$ (see~\cite[5.2.3]{Gr}).   This proves~\eqref{aver5}.

Next, combining~\eqref{aver4} and~\eqref{aver5} and using Minkowski's inequality, we derive
\begin{equation}\label{aver8}
  \begin{split}
     \d^\a \Vert (-\D)^{\a/2} P\Vert_p&\lesssim \Vert f-P\Vert_p+\bigg\Vert \d^{-d}\int_{|h|\le \d} \D_h^\a f(\cdot -\a h) dh \bigg\Vert_p\\
     &\lesssim\d^{-d}\int_{|h|\le \d} \Vert \D_h^\a f\Vert_p dh.
  \end{split}
\end{equation}
Finally, collecting~\eqref{aver3}, \eqref{aver4}, and~\eqref{aver8}, we prove the required estimate.
\medskip

To prove~\eqref{aver2} in the cases $d=1$ and $\a>(1/p-1)_+$, one should repeat step-by-step the proof of Theorem~2.1 from~\cite{K17} with the help of the Jackson inequality~\eqref{JacksonSO-pr} and Lemma~\ref{lemNSB}.
Equivalence~\eqref{aver2+} was obtained in~\cite[8.2.9]{TB} for periodic functions. The proof on $\R^d$ is similar using the fact that ${\xi_j}/{|\xi|}$ is the Fourier multiplier in $L_p(\R^d)$.\hfill$\square$


\subsection*{Proof of Property~\ref{pr6}}
The proof of~\eqref{eq.lemMarchaudMod-pr} is a combination of the direct and inverse estimates
\eqref{JacksonSO-pr}
and
(\ref{eqconverseMod}) as well as the relation $\w_\a(f,\d)_p\asymp \w_\a(f,2\d)_p$ (see (\ref{Run0})).
Note that $E_0(f)_p\le \Vert f\Vert_p$.

Inequality~\eqref{eq.lemMarchaudModinf} follows from~\eqref{eq.lemMarchaudMod-pr} by replacing $f\to f(\e\cdot)$ and $\d\to \frac\d\e$, changing the variables and letting $\e\to\infty$. On the other hand, \eqref{eq.lemMarchaudModinf}  clearly implies~\eqref{eq.lemMarchaudMod-pr} since
$\w_{\g+\a}(f,t)_p \lesssim\Vert f\Vert_p$. \hfill$\square$

\subsection*{Proof of Property~\ref{pr6+}}
Inequality~\eqref{propodmod} follows from~\eqref{binom}. Inequality~\eqref{SJJ} follows from~\cite{ddt} and~\eqref{kwil}.
\hfill$\square$

\subsection*{Proof of Property~\ref{pr7}}

This property follows from Theorem~\ref{thMainMod} below taking $m=\gamma$.
We start with the following lemma.
%
%

\begin{lemma}\label{th1UMMod1}
 Let $f\in L_p(\R^d)$, $0<p<q\le \infty$, $\a\in\N\cup
    ((1/q-1)_+,\infty)$, and $\g, m> 0$ be  such that $\a+\g, \a+m, m-\g\in \N\cup
    ((1/p-1)_+,\infty)$.
Then, for all $\d\in (0,1)$, we have
\begin{equation}\label{eqthRealKUMMod1}
\begin{split}
     \w_\a(f,\d)_q\lesssim \d^\a &\(\int_\d^1 \(\frac{\w_{\a+\g}(f,t)_p}{t^{\a+d(\frac1p-\frac1q)}}\)^\theta\frac{dt}{t}
  +\Vert f\Vert_p^\theta\)^\frac1\theta\\
  &\qquad\qquad\qquad+\(\int_0^\d \(\frac{\w_{\a+m}(f,t)_p}{t^{d(\frac1p-\frac1q)}} \)^{q_1}\frac{dt}{t}   \)^\frac1{q_1},
\end{split}
\end{equation}
where
\begin{equation}\label{tau+}
\theta=\left\{
       \begin{array}{ll}
         \min(q,2), & \hbox{$q<\infty;$} \\
         1, & \hbox{$q=\infty$,}
       \end{array}
     \right.\quad q_1=\left\{
       \begin{array}{ll}
         q, & \hbox{$q<\infty;$} \\
         1, & \hbox{$q=\infty.$}
       \end{array}
     \right.
\end{equation}
\end{lemma}

\begin{proof}

The proof is based on the Marchaud inequality given by~\eqref{eq.lemMarchaudMod-pr}
 and the following Ulyanov type inequalities (see~\cite[Theorem~7.2]{diti})
\begin{equation}\label{diti2}
          \w_{\a+m}(f,\d)_q\lesssim \(\int_0^\d \(
   \frac{\w_{\a+m}(f,t)_p}{t^{d(\frac1p-\frac1q)}}
 \)^{q_1}\frac {dt}{t}\)^\frac 1{q_1}
\end{equation}
and
\begin{equation}\label{-diti1++}
  \Vert f\Vert_q\lesssim \(\int_0^1 \(\frac{\w_{\a+m}(f,t)_p}{t^{d(\frac1p-\frac1q)}}\)^{q_1}\frac{dt}{t}\)^\frac1{q_1}+\Vert f\Vert_p.
\end{equation}
In the latter estimate, we have taken into account the fact that Jackson's inequality~\eqref{JacksonSO-pr} holds for any $0<p\le \infty$.
Using~\eqref{eq.lemMarchaudMod-pr} and~\eqref{diti2}, we derive
\begin{equation*}\label{eqUlRealKforM2}
\begin{split}
  \w_{\a}(f,\d)_q
&\lesssim \d^\a
\(\int_\d^1
u^{-\a \t}
\(
\int_0^u \( \frac{\w_{\a+m}(f,t)_p}{t^{d(\frac1p-\frac1q)}} \)^{q_1}\frac{dt}{t}
\)^{\frac\t{q_1}}
\frac{du}u+\Vert f\Vert_q^\t
\)^\frac1\t.
\end{split}
\end{equation*}
It remains to apply the weighted Hardy inequality for averages (see, e.g.,~\cite[p.~124]{BeSh}) and take into account that
in light of~\eqref{-diti1++}, we have  
\begin{equation*}
  \begin{split}
      \Vert f\Vert_q
       &\lesssim \(\int_0^\d \(\frac{\w_{\a+m}(f,t)_p}{t^{d(\frac1p-\frac1q)}}\)^{q_1}\frac{dt}{t}\)^\frac1{q_1}+\(\d^{\a\t}\int_\d^1 \(\frac{\w_{\a+m}(f,t)_p}{t^{\a+d(\frac1p-\frac1q)}}\)^\theta\frac{dt}{t}+\Vert f\Vert_p^\theta\)^\frac1{\theta}.
   \end{split}
\end{equation*}
\end{proof}

\begin{corollary}\label{cor13.18}
  Under the conditions of Lemma~\ref{th1UMMod1}, we have
\begin{equation}\label{qqqqqqqqqqqqqqwwwwwwwwwwwwwwweeeeeeeeeeeeerrrrrrrrrrr}
\begin{split}
         &\w_{\a}(f,\d)_q\lesssim \left(\int_0^\d\bigg(\frac{\w_{\a+m}(f,t)_p}{t^{d(\frac1p-\frac1q)}}\bigg)^{q_1}\frac{dt}{t}\right)^{\frac1{q_1}}+\d^\a\Vert f\Vert_p\\
&+
\frac{\w_{\a+\g}(f,\d)_p}{\d^{\g+d(\frac1p-1)_+}}\left\{
                                          \begin{array}{ll}
                                            1, & \hbox{$\g>d\(\frac1p-\frac1q\)-d\(\frac1p-1\)_+$}; \\
                                            \ln^{1/\theta}\(\frac1\d+1\), & \hbox{$\g=d\(\frac1p-\frac1q\)-d\(\frac1p-1\)_+$}; \\
                                            \(\frac1\d\)^{d(\frac1p-\frac1q)-d(\frac1p-1)_+-\g}, & \hbox{$\g<d\(\frac1p-\frac1q\)-d\(\frac1p-1\)_+$}.
                                          \end{array}
                                        \right.
\end{split}
\end{equation}

\end{corollary}
The proof of the corollary immediately follows from  inequality~\eqref{eqthRealKUMMod1} and Property~\ref{pr4}.

\medskip

 Now we present a slightly stronger version of the sharp Ulyanov inequality than the one given in~\eqref{ulpr}.
\begin{theorem}\label{thMainMod}
    Let $f\in L_p(\R^d)$, $0<p<q\le\infty$,  $\a \in \N\cup ((1-1/q)_+,\infty)$, and $\g,m\ge 0$ be such that $\a+\g, \a+m, m-\g\in \N\cup ((1/p-1)_+,\infty)$.
    Then, for any $\d \in (0,1)$, we have
    \begin{equation}\label{eqlemMM1}
        \w_\a(f,\d)_q\lesssim \d^\a \Vert f\Vert_p+\frac{\w_{\a+\g}(f,\d)_p}{\d^\g}\eta\(\frac 1\d\)+\(\int_0^\d   \(  \frac{\w_{\a+m}(f,t)_p}{t^{d(\frac1p-\frac1q)}} \)^{q_1}\frac {dt}{t}\)^\frac 1{q_1},
    \end{equation}
where the function $\eta$ is defined in~\eqref{ulpr1} and~\eqref{ulpr2}.
\end{theorem}


%
%

\begin{proof}
The case $\g=0$ is given in~\eqref{diti2}. If $\g>0$, we devide the proof into two cases: $0<p\le 1$ and $p>1$.
First, consider the case $0<p\le 1$ and $p<q\le \infty$.
Then we apply inequality~\eqref{qqqqqqqqqqqqqqwwwwwwwwwwwwwwweeeeeeeeeeeeerrrrrrrrrrr} in the following form:
\begin{equation*}
\begin{split}
  \w_\a(f,\d)_q\lesssim \d^\a \Vert f\Vert_p&+\(\int_0^\d \(\frac{\w_{\a+m}(f,t)_p}{t^{d(\frac1p-\frac1q)}} \)^{q_1}\frac{dt}{t}   \)^\frac1{q_1}\\
&+\frac {\w_{\a+\g}(f,\d)_p}{\d^{\g+d(\frac1p-1)}}\left\{
                                          \begin{array}{ll}
                                            1, & \hbox{$\g>d\(1-\frac1q\)$;} \\
                                            \ln^{1/\theta}\(\frac1\d+1\), & \hbox{$\g=d\(1-\frac1q\)$;} \\
                                            \(\frac1\d\)^{d(1-\frac1q)-\g}, & \hbox{$0<\g<d\(1-\frac1q\)$.}
                                          \end{array}
                                        \right.
\end{split}
\end{equation*}
Hence, \eqref{eqlemMM1} holds for the following parameters: 

\begin{itemize}

  \vspace{2mm} \item[(i)] $\g>0$ and $0<q\le 1$;

 \vspace{2mm}
  \item[(ii)] $\g=d(1-1/q)$ and $1<q\le 2$ with $\a+\g\not\in \N$;

 \vspace{2mm}  \item[(iii)] $\g=d(1-1/q)$ and $q=\infty$ with $\a+\g\not\in \N$;

 \vspace{2mm} \item[(iv)]  $0<\g<d\(1-1/q\)_+$.

\end{itemize}

\noindent Thus, in the case $0<p\le 1$ it remains to consider:

 \vspace{2mm}
(v) $\g=d\(1-1/q\)\ge 1$, $\a+\g\not\in \N$, and $2<q<\infty$
(note that for such $q$ and $d\ge 2$ we always have $d(1-1/q)\ge 1$);

 \vspace{2mm}
(vi) $\g=d\(1-1/q\)\ge 1$, $1<q\le \infty$, and $\a+\g\in \N$;

\vspace{2mm}
(vii) $0<\g=d\(1-\frac1q\)_+=1$ and $d=1$.


To obtain~\eqref{eqlemMM1} in the cases (v)--(vii), we will use the Hardy-Littlewood-Nikolskii inequalities proved in Subsection~\ref{sec5} and the realizations of the $K$-functionals from Property~\ref{pr10}.


Let $P_\s\in \mathcal{B}_{\s,p}$ be such that
\begin{equation}\label{eqth1.1Kd0--}
\Vert f-P_\s\Vert_p\lesssim E_\s(f)_p.
\end{equation}
From~\cite[Lemma 4.2]{diti} and Jackson's inequality~\eqref{JacksonSO-pr}, 
it follows that 
\begin{equation}\label{eqth++}\begin{split}
\Vert f-P_{2^n}\Vert_q
&\lesssim
 \(\sum_{\nu=n}^\infty{2^{\nu q_1
d(\frac1p-\frac1q)}}\Vert
f-P_{2^{\nu}}\Vert_p^{q_1}\)^{\frac1{q_1}}\\
&\lesssim
 \(\sum_{\nu=n}^\infty{2^{\nu q_1
d(\frac1p-\frac1q)}}\w_{\a+m}(f,2^{-\nu})_p^{q_1}\)^{\frac1{q_1}}.
\end{split}
\end{equation}
Thus, taking into account realization~\eqref{eqvwithbest-pr} as well as (\ref{eqth1.1Kd0--}) and (\ref{eqth++}), we obtain
\begin{equation}\label{GGGI}
\begin{split}
\w_{\a}(f,2^{-n})_q
        &\lesssim \Vert f-P_{2^{n}}\Vert_q+2^{-\a n}\sup_{|\zeta|=1,\,\zeta\in\R^d}\Vert D_{\zeta}^{\a} P_{2^n}\Vert_{q}\\
        &\lesssim \(\sum_{\nu=n}^\infty \({2^{\nu
d(\frac1p-\frac1q)}}\w_{\a+m}(f,2^{-\nu})_p\)^{q_1}\)^{\frac1{q_1}}+2^{-\a n}\sup_{|\zeta|=1,\,\zeta\in\R^d}\Vert D_{\zeta}^{\a} P_{2^n}\Vert_{q}.
\end{split}
 \end{equation}

Now, considering the case (v), we estimate the first summand in the last inequality using Lemma~\ref{lemma+} and~\eqref{eqvwithbest-pr}. We derive
\begin{equation}\label{GGGI1}
  \begin{split}
    \sup_{|\zeta|=1,\,\zeta\in\R^d}\Vert D_{\zeta}^{\a+\g} P_{2^n}\Vert_{q}\lesssim 2^{\a n} 2^{d(\frac1p-1)n}n^{\frac1q}\w_{\a+\g}(f,2^{-n})_p+\Vert P_{2^n}\Vert_q.
  \end{split}
\end{equation}
To estimate $\Vert P_{2^n}\Vert_q$, we use the same arguments as in the proof of Lemma~\ref{th1UMMod1} and Corollary~\ref{cor13.18}.  We have
\begin{equation}\label{GGGI2}
  \begin{split}
    \Vert P_{2^n}\Vert_q&\lesssim \Vert f\Vert_q + E_{2^n}(f)_q\lesssim \Vert f\Vert_q\\
    &\lesssim \(\int_0^1 \(\frac{\w_{\a+\g}(f,t)_p}{t^{d(\frac1p-\frac1q)}}\)^{q_1}\frac{dt}{t}\)^\frac1{q_1}+\Vert f\Vert_p\\
       &\lesssim \(\int_0^\d\(\frac{\w_{\a+m}(f,t)_p}{t^{d(\frac1p-\frac1q)}}\)^{q_1}\frac{dt}{t}\)^\frac1{q_1}+
       \(\int_\d^1\(\frac{\w_{\a+\g}(f,t)_p}{t^{\a+d(\frac1p-\frac1q)}}\)^{q_1}\frac{dt}{t}+\Vert f\Vert_p^{q_1}\)^\frac1{q_1}\\
&\lesssim \left(\int_0^\d\bigg(\frac{\w_{\a+m}(f,t)_p}{t^{d(\frac1p-\frac1q)}}\bigg)^{q_1}\frac{dt}{t}\right)^{\frac1{q_1}}+\Vert f\Vert_p+\frac{\w_{\a+\g}(f,\d)_p}{\d^{\g+d(\frac1p-1)}}\ln^{\frac 1{q_1}}\(\frac1\d+1\).
\end{split}
\end{equation}
Thus, combining~\eqref{GGGI}--\eqref{GGGI2}, we arrive at~\eqref{eqlemMM1} in the case (v).

To consider the case (vi), we use similar arguments applying Lemma~\ref{lemPolSob}. Remark that in this case we do not need to estimate $\Vert f\Vert_q$.




In the case (vii), the required estimate follows from the equality
$$
P_\s(x)=\int_{-\infty}^x P_\s'(t)dt.
$$
Note that by the Nikolskii inequality~\eqref{eqNIKNIK}, $P_\s\in L_1(\R^d)$ whenever $P_\s\in L_p(\R^d)$ and using~\cite[Theorem~3.2.5]{nikol-book} we have $P(x)\to 0$ as $x\to \infty$. See also~\cite[p.~118]{PP}.

%

Let $1<p<q< \infty$. In this case the proof is similar. First, we use inequality~\eqref{qqqqqqqqqqqqqqwwwwwwwwwwwwwwweeeeeeeeeeeeerrrrrrrrrrr} as follows
\begin{equation*}
\begin{split}
            \w_{\a}(f,\d)_q&\lesssim \left(\int_0^\d\bigg(\frac{\w_{\a+m}(f,t)_p}{t^{d(\frac1p-\frac1q)}}\bigg)^{q_1}\frac{dt}{t}\right)^{\frac1{q_1}}+\d^\a\Vert f\Vert_p\\
&+\frac {\w_{\a+\g}(f,\d)_p}{\d^{\g}}\left\{
                                          \begin{array}{ll}
                                            1, & \hbox{$\g>d\(\frac1p-\frac1q\)$;} \\
                                            \ln^{1/\theta}\(\frac1\d+1\), & \hbox{$\g=d\(\frac1p-\frac1q\)$;} \\
                                            \(\frac1\d\)^{d(\frac1p-\frac1q)-\g}, & \hbox{$0<\g<d\(\frac1p-\frac1q\)$.}
                                          \end{array}
                                        \right.
\end{split}
\end{equation*}

Hence, \eqref{eqlemMM1} holds in all cases except the case $\g=d(\frac1p-\frac1q)$.
In the latter case, if $q<\infty$,  the Hardy-Littlewood inequality $\Vert f\Vert_q\lesssim \Vert (-\Delta)^{\g/2} f\Vert_p$ and~\eqref{GGGI} yield~\eqref{eqlemMM1}.  Indeed,  by Corollary~\ref{lemRAZZZ}, we have that
$$
\sup_{|\zeta|=1,\,\zeta\in\R^d}
\Vert D_{\zeta}^{\a} P_\s\Vert_{q}\asymp \Vert (-\D)^{\a/2}P_\s\Vert_q
$$
and
$$
\sup_{|\zeta|=1,\,\zeta\in\R^d}
\Vert D_\zeta^{\a+\g} P_\s\Vert_{p}\asymp \Vert (-\D)^{(\a+\g)/2}P_\s\Vert_p.
$$
It remains to apply~\eqref{eqvwithbest-pr}.

Finally, to show \eqref{eqlemMM1} in the case $1<p<q=\infty$ and $\g=d/p$, we note that by Lemma~\ref{lemma+inf} and~\eqref{eqvwithbest-pr}, we derive
$$
    \sup_{|\zeta|=1,\,\zeta\in\R^d}\Vert D_{\zeta}^{\a} P_{2^n}\Vert_{\infty}\lesssim 2^{(\a+\g) n} n^{\frac1{p'}}\w_{\a+\g}(f,2^{-n})_p+\Vert f\Vert_p.
$$
This together with~\eqref{GGGI} implies~\eqref{eqlemMM1}. 
\end{proof}

\begin{corollary}\label{important corollary} Let
 $0<p\le 1<q\le\infty$ and $d\ge 2$. We have
\begin{equation*}
  \w_\a(f,\d)_q\lesssim \(\int_0^\d   \(
   \frac{\w_{\a+d(1-\frac1q)}(f,t)_p}{t^{d(\frac1p-\frac1q)}}
 \)^{q_1}\frac {dt}{t}\)^\frac 1{q_1}
\end{equation*}
   provided that $\a+d(1-1/q)\in\N$.
In particular, for any $d, \a\in\N$,
     we have
    \begin{equation*}
        \w_\a(f,\d)_\infty\lesssim
        \int_0^\d   \frac{\w_{\a+d}(f,t)_1}{t^{d}} \frac {dt}{t}.
    \end{equation*}
\end{corollary}

\bigskip

Finally, we note that unlike the case of Lebesgue spaces, the sharp Ulyanov inequality in Hardy spaces has the same form both in quasi-Banach spaces ($0<p<1$) and Banach spaces ($p\ge 1$).
Recall that the real Hardy spaces $H_p(\R^d)$,
$0<p<\infty$,\index{\bigskip\textbf{Spaces}!$H_p(\R^d)$}\label{HPR} is the class of tempered distributions  $f\in
\mathscr{S}'(\R^d)$ such that
$$
\Vert f\Vert_{H_p}=\Vert f\Vert_{H_p(\R^d)}=\big\Vert
\sup_{t>0}|\vp_t*f(x)|\big\Vert_{L_p(\R^d)}<\infty,
$$
where $\vp\in \mathscr{S}(\R^d)$, $\widehat{\vp}(0)\neq 0$, and
$\vp_t(x)=t^{-d}\vp(x/t)$ (see~\cite[Ch. III]{Stein93}).
The 
moduli of smoothness in $H_p(\R^d)$ are defined by
$
\w_\a(f,\d)_{H_p}:=\sup_{|h|<\d}\Vert \D_h^\a f\Vert_{H_p},
$\index{\bigskip\textbf{Functionals and functions}!$\w_\a(f,\d)_{H_p}$, modulus of smoothness in the Hardy spaces}\label{MODAHP}
where the fractional difference $\D_h^\a f$ is given by~\eqref{def-mod++}.
Using boundedness properties  of the  fractional integrals in the Hardy spaces,  the following sharp Ulyanov inequality in~$H_p(\R^d)$ was derived in~\cite{KT19m}.

\begin{theorem}\label{realHpsharpURd}  Let $f\in H_p(\R^d)$, $0<p<q<\infty$, $\a\in \N \cup ((1/q-1)_+,\infty)$ and $\a+\t \in \N \cup ((1/p-1)_+,\infty)$, and $\theta=d(1/p-1/q)$. Then, for any  $\d\in (0,1)$, we have
    \begin{equation*}
        \w_\a(f,\d)_{H_q}\lesssim
        \left(\int_0^\d\bigg(\frac{\w_{\a+\theta}(f,t)_{H_p}}{t^{\theta}}\bigg)^q\frac{dt}{t}\right)^{\frac1q}.
    \end{equation*}
\end{theorem}

\medskip

\subsection*{Proof of Property~\ref{Koly}} See~\cite[Theorem~2.6]{Treb}.

The following generalization of~\eqref{eqth3.1Kmod} for Hardy spaces has been recently obtained in~\cite[Theorem~15]{KT19m}.

\begin{theorem}
Let $f\in H_p(\R^d)$, $0<p< q<\infty$,
$\theta=d\(1/p-1/q\)$, and $\a\in \N\cup ((1/p-1)_+,\infty)$, $\a>\t$. Then
\begin{equation*}
    \d^{\a-\theta}\(\int_\d^\infty
\(\frac{\w_\a(f,t)_{H_q}}{t^{\a-\theta}}\)^p\frac{dt}{t}\)^\frac1p \lesssim\(\int_0^\d \(\frac{\w_\a(f,t)_{H_p}}{t^\t}\)^q\frac{dt}{t}\)^\frac1q.
\end{equation*}
\end{theorem}

\medskip

\subsection*{Proof of Property~\ref{pr8}}

See~\cite[Theorem~12.2]{DoTi} for the proof of inequality~\eqref{inequalTrebels2} and~\cite[Theorem 2.3]{Treb} for inequality~\eqref{inequalTrebels1}.
%
%
%

\subsection*{Proof of Property~\ref{pr11}}
The proof of the Jackson inequality~\eqref{JacksonSO-pr} in the multidimensional case is based on the following lemma.


\begin{lemma}\label{ledr}{\sc (See \cite{BRS09}.)}
Let $f\in L_p(\R)$, $0<p<1$, $r\in \N$, and $\s>0$. Then
\begin{equation*}
  \frac1{2\s}\int_{[-\s,\s]} \Vert f-V_{\s,\l}(f)\Vert_{L_p(\R)}^p d\l\lesssim \w_{2r+1}(f,\s^{-1})_{L_p(\R)}^p
\end{equation*}
and
\begin{equation*}
  \frac1{2\s}\int_{[-\s,\s]} \Vert V_{\s,\l}(f)\Vert_{L_p(\R)}^p d\l\lesssim \Vert f\Vert_{L_p(\R)}^p,
\end{equation*}
where
$$
V_{\s,\l}(f)(x)=2\pi \sum_{k\in \Z} f\(\frac k\s+\l\) (\mathcal{F}\vp)(\s(x-\l)-k),\quad
\vp(\xi)=(1+i\xi^{2r+1})\eta(\xi),
$$
and $\eta\in C^\infty(\R)$, $\eta(\xi)=1$ for $|\xi|\le 1/2$, and $\eta(\xi)=0$ for $|x|\ge 1$.
\end{lemma}

%

\begin{proof}[Proof of inequality~\eqref{JacksonSO-pr}]

We consider only the case $0<p<1$. The proof of the inequality in the case $p\ge 1$ can be found, e.g.,  in~\cite[5.3.2]{timan}.

First, we consider the case $\a=r\in \N$. 
Denote
$$
V_{\s,\l,j}(f)(x)=2\pi \sum_{k\in \Z} f\(x_1,\dots,x_{j-1},\frac k\s+\l,x_{j+1},\dots,x_d\) (\mathcal{F}\vp)(\s(x_j-\l)-k).
$$
Then, we have
\begin{equation*}
  \begin{split}
      E_{d^{-\frac12}\s}(f)_p^p \le \frac1{(2\s)^d}\int_{[-\s,\s]^d} \Vert f- V_{\s,\l_1,1} \circ\dots\circ V_{\s,\l_d,d}(f)\Vert_p^p d\l_1\dots d\l_d.
   \end{split}
\end{equation*}
Using the equality
\begin{equation*}
  \begin{split}
     f- &V_{\s,\l_1,1} \circ\dots\circ V_{\s,\l_d,d}(f)=f-V_{\s,\l_1,1}(f)\\
     &+V_{\s,\l_1,1}(f-V_{\s,\l_2,2}(f))+\dots + V_{\s,\l_1,1} \circ\dots\circ V_{\s,\l_{d-1},d-1}(f-V_{\s,\l_d,d}(f)),
   \end{split}
\end{equation*}
Lemma~\ref{ledr}, and Fubini's theorem, we derive
\begin{equation*}
  \begin{split}
     &\frac1{(2\s)^d}\int_{[-\s,\s]^d} \Vert f- V_{\s,\l_1,1} \circ\dots\circ V_{\s,\l_d,d}(f)\Vert_p^p d\l_1,\dots d\l_d\\
     &\le \frac1{2\s}\int_{[-\s,\s]} \Vert f- V_{\s,\l_1,1}(f)\Vert_p^p d\l_1+\frac1{(2\s)^2}\int_{[-\s,\s]^2} \Vert V_{\s,\l_1,1}(f-V_{\s,\l_2,2}(f))\Vert_p^p d\l_1 d\l_2\\
     &\qquad\quad+\dots+\frac1{(2\s)^d}\int_{[-\s,\s]^d} \Vert  V_{\s,\l_1,1} \circ\dots\circ V_{\s,\l_{d-1},d-1}(f-V_{\s,\l_d,d}(f)) \Vert_p^p d\l_1\dots d\l_d\\
     &\lesssim \sum_{j=1}^d \frac1{2\s} \int_{[-\s,\s]} \Vert f- V_{\s,\l_j,j}(f)\Vert_p^p d\l_j\lesssim \sum_{j=1}^d \w_{2r+1}^{(j)} (f,\s^{-1})_p^p,
   \end{split}
\end{equation*}
where  $\w_{2r+1}^{(j)} (f,\s^{-1})_p$ is the partial modulus of smoothness given in~\eqref{moddd}.

Thus, using the fact that
$\w_{2r+1}^{(j)}(f,d^{1/2}\d)_p\lesssim \w_{2r+1}^{(j)}(f,\d)_p$, see, e.g.,~\cite[p.~370]{DL}, and relations~\eqref{propodmod} and~\eqref{mix1},
we have
\begin{equation*}
  E_{\s}(f)_p\lesssim \sum_{j=1}^d \w_{2r+1}^{(j)}(f,d^\frac12\s^{-1})_p\lesssim   \sum_{j=1}^d \w_{r}^{(j)}(f,\s^{-1})_p   \lesssim \w_{r}(f,\s^{-1})_p.
\end{equation*}
Finally, using the above estimate for $r=\a+k\in \N$ with $k>(1/p-1)_+$ and Property~\ref{pr6+}, we get
$$
  E_\s(f)_p\lesssim \w_{\a+k}\(f,\s^{-1}\)_p\lesssim \w_{\a}\(f,\s^{-1}\)_p.
$$
\end{proof}


\subsection*{Proof of Property~\ref{pr12}}
The proof of~\eqref{eqconverseMod} follows the standard argument using the Littlewood-Paley decomposition in the case $1<p<\infty$, see~\cite{ddt}, and telescoping sums in the cases $0<p\le 1$ and $p=\infty$, see, e.g.,~\cite{KT19m}. Since we deal with fractional moduli of smoothness, we also need the inequality
$\w_\a(P_\s,\d)_p\lesssim \s^\a \Vert P_\s\Vert_p$, $0<\d\le {\pi}/{\s}$,
which follows from~\eqref{ineqNS3cor} and~\eqref{ineqNS3corBEr}.\hfill$\square$



\subsection*{Proof of Property~\ref{Rathor}}
In the case $0<p<1$, the property can be proved repeating step-by-step the proof of Theorem~1 from~\cite{K07} using also inequalities~\eqref{Run0}, \eqref{JacksonSO-pr}, and~\eqref{eqconverseMod}. For the case $p\ge 1$ see~\cite[Theorem~8.2]{GIT}.
\hfill$\square$

\subsection*{Proof of Property~\ref{pr13}}
%
%
In the case $0<p<1$, the proof of the first inequality in~\eqref{eq7R} easily follows from~\eqref{JacksonSO-pr} and~\eqref{ineqNS3}:
\begin{equation*}
  2^{-n\a} \sup_{|\zeta|=1,\, \zeta\in \R^d}\Vert D_\zeta^\a P_{2^n} \Vert_p
\lesssim \omega_\a(P_{2^n},2^{-n})_{p}\lesssim \Vert f-P_{2^n}\Vert_p+\omega_\a(f,2^{-n})_{p}\lesssim \omega_\a(f,2^{-n})_{p}.
\end{equation*}
To obtain the second inequality in~\eqref{eq7R}, one needs to apply~\cite[Theorem~2.1]{KT19a}, taking into account inequalities~\eqref{JacksonSO-pr} and~\eqref{ineqNS3}.

For the case $1<p<\infty$ see~\cite[Theorem~6.1]{KT19a}.\hfill$\square$



\subsection*{Proof of Property~\ref{pr9}}
The estimate $\gtrsim$ follows from the realization result given in~\eqref{eq.th6.0d-pr}. To obtain the estimate $\lesssim$, we use the fact that the function
$$
m(\xi)=\(\frac{1-e^{i(\xi_1+\dots+\xi_d)} }{\xi_1+\dots+\xi_d}\)^\a
$$
is a Fourier multipliers in $L_p(\R^d)$ for all $1\le p\le \infty$, see, e.g.,~\cite[Proposition~1]{Wil}. Then by standard arguments, we derive for any function $g\in L_p(\R^d)$
\begin{equation*}
  \begin{split}
      {\omega}_\a(f,\d)_p \lesssim {\omega}_\a(f-g,\d)_p + {\omega}_\a(g,\d)_p  \lesssim \Vert f-g \Vert_p+\d^\a\sup_{|\zeta|=1,\, \zeta\in \R^d}\Vert D_\zeta^\a g \Vert_p.
   \end{split}
\end{equation*}
Taking the infimum over all $g$, we obtain the estimate $\lesssim$ in~\eqref{eq.th6.0d++pr}.\hfill$\square$

\subsection*{Proof of Property~\ref{pr10}}

First, let us prove the estimate from above in~\eqref{eq.th6.0d-pr}. Let $P_\s\in\mathcal{B}_{\s,p}$, $\d\in (0,1/\s)$.
By Corollary~\ref{corNSB}, we estimate
\begin{equation*}
  \begin{split}
      \w_\a(f,\d)_p&\lesssim \Vert f-P_\s\Vert_p+\w_\a(P_\s,\d)_p\\
      &\lesssim \Vert f-P_\s\Vert_p+\d^{\a}\sup_{\zeta\in \R^d,\,|\zeta|=1}\Vert D_{\zeta}^{\a} P_\s\Vert_{p}.
   \end{split}
\end{equation*}
It remains to take infimum over all  $P_\s\in\mathcal{B}_{\s,p}$.

Now let $P_\s\in \mathcal{B}_{\s,p}$ be such that $\Vert f-P_\s\Vert_p\lesssim E_\s(f)_p$.
Then by the Jackson inequality~\eqref{JacksonSO-pr} and the Nikolskii-Stechkin inequality~\eqref{ineqNS3cor}, we obtain
\begin{equation*}
  \begin{split}
     \w_{\a}(f,\s^{-1})_p &\lesssim \mathcal{R}_\a(f,\s^{-1})_p\lesssim \Vert f-P_\s\Vert_p+\s^{-\a}\sup_{\zeta\in \R^d,\,|\zeta|=1}\Vert D_{\zeta}^{\a} P_\s\Vert_{p}\\
     &\lesssim \w_{\a}(f,\s^{-1})_p+\w_{\a}(P_\s,\s^{-1})_p\lesssim \w_{\a}(f,\s^{-1})_p,
   \end{split}
\end{equation*}
that is, \eqref{eq.th6.0d-pr} and~\eqref{eqvwithbest-pr} follow.


\medskip

Similarly, taking into account Corollary~\ref{eq++}, we can prove~\eqref{RealW}.\hfill$\square$

\bigskip
{\bf{Acknowledgements.}}
The first author was partially supported by DFG project KO 5804/1-1.
The second author was partially supported by MTM 2017-87409-P, 2017 SGR 358, and by the CERCA Programme of the Generalitat de Catalunya.
Part of the work was done during the visit of the authors to the Isaac Newton Institute for Mathematical Sciences, EPSCR Grant no EP/K032208/1.
We would like to thank O.~Dom\'inguez and D.~Gorbachev for bringing our attention to the Brezis-Wainger embedding with respect to Hardy--Littlewood--Nikolskii's inequalities. More results in this direction can be seen in \cite{oscar}.

\end{document}